\documentclass{amsart}      % Comments after  % are ignored
\usepackage{amssymb,amsthm, amsmath, amsfonts} % Typical maths resource packages
\usepackage{graphics}                 % Packages to allow inclusion of graphics
\usepackage{color}                    % For creating coloured text and background
\usepackage{hyperref}                 % For creating hyperlinks in cross references
\usepackage[all]{xy}
\usepackage{enumerate}

\newtheorem{theorem}{Theorem}[section]
\newtheorem{lemma}[theorem]{Lemma}
\newtheorem{proposition}[theorem]{Proposition}

\newtheorem{corollary}[theorem]{Corollary}
\newtheorem{definition}[theorem]{Definition}
\newtheorem{example}[theorem]{Example}

\numberwithin{equation}{section}

\DeclareMathOperator{\rad}{rad}
\DeclareMathOperator{\Ob}{Ob}

\DeclareMathOperator{\Mor}{Mor}
\DeclareMathOperator{\Hom}{Hom}
\DeclareMathOperator{\End}{End}
\DeclareMathOperator{\Top}{Top}
\DeclareMathOperator{\Soc}{Soc}

\DeclareMathOperator{\Stab}{Stab}
\DeclareMathOperator{\Char}{char}

\title{On the representation types of category algebras of finite EI categories}
\author{Liping Li}
\address{School of Mathematics, University of Minnesota, Minneapolis, MN, 55455, USA}
\email{lixxx480@math.umn.edu}
\thanks{The author would like to thank his thesis advisor, Professor Peter Webb, for the proposal to study this problem, and for many invaluable suggestions in numerous discussions. He wants to express appreciation to Professor Edward Green, Professor Markus Schmidmeier and Dr. Ryan Kinser, from whom the author learned many useful techniques applied in this paper. He also thanks the referee whose suggestions make this paper much more concise.}
\begin{document}

\begin{abstract}
A finite EI category is a small category with finitely many morphisms such that every endomorphism is an isomorphism. They include finite groups, finite posets and free categories of finite quivers as special cases. In this paper we consider the representation types of finite EI categories with two or three objects, describe some criteria for finite representation type, and use them to classify the representation types of the category algebras of several classes of finite EI categories with extra properties.
\end{abstract}
\maketitle

\section{Introduction}

\textit{Finite EI categories} are small categories with finitely many morphisms such that every endomorphism is an isomorphism. They are studied in different areas as topology (\cite{Dieck, Luck}), representation theory (\cite{Dietrich1, Dietrich2, Li1, Li2, Webb2, Webb3, Xu2}), homological algebra (\cite{Webb1, Xu1, Xu3}), and fusion systems (\cite{AKO, Linckelmann}). Interesting examples include finite groups (categories with one object), finite posets, and free categories of finite quivers.

For a fixed algebraically closed field $k$, a finite EI category $\mathcal{C}$ determines a finite-dimensional $k$-algebra $k\mathcal{C}$ called the \textit{category algebra}. It is said to be of \textit{finite representation type} if $k\mathcal{C}$-mod has only finitely many indecomposable objects up to isomorphism. Otherwise, it is said to be of \textit{infinite representation type}. In this paper we mainly study the representation types of category algebras $k\mathcal{C}$, where $\mathcal{C}$ has only two or three objects. This problem has been considered in Dietrich's thesis (\cite{Dietrich1, Dietrich2}) and some useful criteria of infinite representation type have been acquired.

Let $\mathcal{C}$ be a connected skeletal finite EI category with two objects. In the case that the automorphism groups of both objects have orders invertible in $k$, we can construct the ordinary quiver $Q$ (which is an acyclic finite quiver) of $k\mathcal{C}$ using the algorithm described in \cite{Li1}, and $k\mathcal{C}$ is Morita equivalent to the hereditary algebra $kQ$. Therefore, its representation type can be determined by Gabriel's theorem. For the situation that the automorphism groups of both objects are $p$-groups, where $p = \Char(k)$, the classification has been obtained in \cite{Dietrich2} by applying Bongartz's list described in \cite{Bautista, Bongartz}.

When both automorphism groups in $\mathcal{C}$ are arbitrary, the situation is much more complicated. We introduce some notation here. Let $\mathcal{C}$ be pictured as below, where $G = \mathcal{C} (x, x)$ and $H = \mathcal{C} (y, y)$. Using some reduction we can show that if $k\mathcal{C}$ is of finite representation type, then $\mathcal{C} (x, y)$ has only one orbit as an $(H, G)$-biset. Thus we let $\alpha$ be a fixed morphism in $\mathcal{C} (x, y)$. Define $G_0 = \Stab _G (\alpha)$, $G_1 = \Stab _G (H \alpha)$, $H_0 = \Stab _H (\alpha)$, and $H_1 = \Stab _H (\alpha G)$. Then $G_0 \lhd G_1 \leqslant G$ and $H_0 \lhd G_1 \leqslant H$. Moreover, $G_1 / G_0 \cong H_1 / H_0$ as groups.
\begin{equation*}
\xymatrix{ \mathcal{C}: & x \ar@(ld,lu)[]|{G} \ar @<1ex>[rr] ^{H\alpha G} \ar@<-1ex>[rr] ^{\ldots} & & y \ar@(rd,ru)[]|{H}}.
\end{equation*}

We characterize representations of $\mathcal{C}$ explicitly by linear maps satisfying some property related to representations of the above groups. Therefore, many techniques from group representation theory (for example, induction and restriction, permutation modules, block theory, Brauer graphs, and etc.) can be applied to exploit the representation type of $k\mathcal{C}$. We get the following criteria:

\begin{theorem}
Let $\mathcal{C}$ be as above and $p = \Char (k) \geqslant 0$. If $k\mathcal{C}$ is of finite representation type, then the following conditions must be satisfied:
\begin{enumerate}
\item Both $G$ and $H$ have cyclic Sylow $p$-subgroups.
\item Either $G$ or $H$ acts transitively on $\mathcal{C} (x, y)$, i.e., $\mathcal{C} (x, y)$ equals $H \alpha$ or $\alpha G$.
\item If $p \geqslant 5$, then either $O^{p'} G$ (the subgroup generated by all Sylow $p$-subgroups in $G$) or $O^{p'} H$ acts trivially on $\mathcal{C} (x, y)$. Moreover, for every $p$-subgroup $P \leqslant G$ (resp., $Q \leqslant H$), $P \cap G_0$ is either 1 or $P$ (resp., $Q \cap H_0$ is either 1 or $Q$).
\end{enumerate}
\end{theorem}

Consequently, in many cases the representation type of $k\mathcal{C}$ can be determined from two pieces of information: the transitivity of actions by $G$ and $H$, and the triviality of actions by Sylow $p$-subgroups in $G$ and $H$. Indeed, the combination of conditions (a-c)
\begin{enumerate}[ \indent (a)]
\item Both $G$ and $H$ act transitively.
\item One of $G$ and $H$ acts transitively, and the other one does not.
\item Neither $G$ nor $H$ acts transitively.
\end{enumerate}
and conditions (1-3)
\begin{enumerate}
\item Both $O^{p'}G$ and $O^{p'}H$ act trivially.
\item One of $O^{p'}G$ and $O^{p'}H$ acts trivially, and the other one does not.
\item Neither $O^{p'}G$ nor $O^{p'}H$ acts trivially.
\end{enumerate}
gives us 8 situations (the combination (a)+(2) cannot happen). By Theorem 1.1, if $\Char(k) = p \neq 2, 3$, we get infinite representation type of $k\mathcal{C}$ for five cases: (c)+(1), (c)+(2), (c)+(3), (a)+(3), (b)+(3). The following theorem tells us the finite representation type for case (a)+(1).

\begin{theorem}
Let $\mathcal{C}$ be as in the previous theorem, and suppose that both $G$ and $H$ act transitively on $\mathcal{C} (x, y)$. If $\Char(k) = p \neq 2, 3$, then $k\mathcal{C}$ is of finite representation type if and only if both $O^{p'} G$ and $O^{p'} H$ act trivially on $\mathcal{C} (x, y)$.
\end{theorem}

Although cases (b)+(1) and (b)+(2) are not completely resolved here, by considering the structure of certain permutation modules we are able to obtain the following criteria.

\begin{theorem}
Let $\mathcal{C}$ be as in Theorem 1.1 and suppose without loss of generality that $H$ acts transitively on $\mathcal{C} (x, y)$. If $k\mathcal{C}$ is of finite representation type, then the following conditions must be satisfied:
\begin{enumerate}
\item $\dim_k \End_{kH} (k \uparrow _{H_1}^H) = | H_1 \backslash H /H_1 | \leqslant 3$;
\item all summands of $\Top (k \uparrow _{H_1}^H)$ are not isomorphic;
\item If $M, N$ are indecomposable summands of $k \uparrow _{H_1}^H$, then $\Hom _{kH} (M, N) = 0$.
\end{enumerate}
\end{theorem}

We classify the representation type of $k\mathcal{C}$ for $p \geqslant 5$ when both $G$ and $H$ are abelian.

\begin{theorem}
Let $\mathcal{C}$ be as in Theorem 1.1 and $\Char (k) = p \neq 2,3$. Without loss of generality assume that $H$ acts transitively on $\mathcal{C} (x, y)$. Let $s = | O^{p'}G |$, $t = | O^{p'}H |$, and $n = | H: H_1 |$. If both $G$ and $H$ are abelian, then $k\mathcal{C}$ is of finite representation type if and only if both $O^{p'}G$ and $O^{p'}H$ act trivially on $\mathcal{C} (x, y)$, and one of the following conditions holds:
\begin{enumerate}
\item $n = 1$ for $s, t \geqslant p$, i.e., $H = H_1$;
\item $n \leqslant 2$ for $s = 1, t \geqslant p$ or $t = 1, s \geqslant p$;
\item $n \leqslant 3$ for $t = s = 1$.
\end{enumerate}
\end{theorem}

This paper is organized as follows. Some background knowledge and preliminary reductions on finite EI categories and their representations are included in the second section. In Section 3 we study in details decomposition of bisets and actions of groups on these bisets, and develop criteria contained in Theorems 1.1 and 1.3. Using these results, we determine the representation types of category algebras of most finite EI categories with two objects in Section 4. In the last section we consider some special finite EI categories with three objects.

Throughout this paper $k$ is an algebraically closed field and $k$-vec is the category of finite-dimensional $k$-vector spaces. For a finite EI category $\mathcal{C}$, by $\Ob \mathcal{C}$ and $\Mor \mathcal{C}$ we denote sets of objects and morphisms in $\mathcal{C}$ respectively; $\mathcal{C} (x, y)$ denotes the set of morphisms from $x$ to $y$; and by $\mathcal{C}$-rep (or $k\mathcal{C}$-mod) we denote the category of representations of $\mathcal{C}$. For a module $M$, $\Top (M)$ and $\Soc (M)$ mean the top and the socle of $M$ respectively. All modules we consider are finitely generated left modules. Composition of morphisms and maps is from right to left. To deal with all characteristics in a uniform way, when $p = 0$, we view the trivial group $1$ as the Sylow 0-subgroup of a finite group.

\section{Preliminaries}

For the reader's convenience, we include in this section some background on the representation theory of finite EI categories. Please refer to \cite{Webb2, Webb3, Xu1, Xu2} for more details.

A \textit{finite EI category} $\mathcal{C}$ is a small category with finitely many morphisms such that every endomorphism in $\mathcal{C}$ is an isomorphism. Examples of finite EI categories include finite groups (viewed as categories with one object), finite posets (all endomorphisms are identities), and \textit{orbit categories} (see \cite{Webb3}). The category $\mathcal{C}$ is \textit{connected} if for any two distinct objects $x$ and $y$, there is a list of objects $x=x_0, x_1, \ldots, x_n=y$ such that either $\mathcal{C} (x_i, x_{i+1})$ or $\mathcal{C} (x_{i+1}, x_i)$ is not empty, $0 \leqslant i \leqslant n-1$. Every finite EI category is a disjoint union of connected components, and each component as a full subcategory is again a finite EI category.

A \textit{representation} of $\mathcal{C}$ is a covariant functor $R: \mathcal{C} \rightarrow k$-vec. The functor $R$ assigns a vector space $R(x)$ to each object $x$ in $\mathcal{C}$, and a linear transformation $R(\alpha): R(x) \rightarrow R(y)$ to each morphism $\alpha: x \rightarrow y$ such that all composition relations of morphisms in $\mathcal{C}$ are preserved under $R$. A \textit{homomorphism} $\varphi: R_1 \rightarrow R_2$ of two representations is a natural transformation of functors.

A finite EI category $\mathcal{C}$ determines a finite-dimensional $k$-algebra $k\mathcal{C}$ called the \textit{category algebra}. It has a basis constituted of all morphisms in $\mathcal{C}$, and the multiplication is defined by the composition of morphisms (the composite is 0 when two morphisms cannot be composed) and bilinearity. By Theorem 7.1 of \cite{Mitchell}, a representation of $\mathcal{C}$ is equivalent to a $k\mathcal{C}$-module. Thus we do not distinguish these two concepts throughout this paper.

By Proposition 2.2 in \cite{Webb2}, if $\mathcal{C}$ and $\mathcal{D}$ are equivalent finite EI categories, $k \mathcal{C}$-mod is Morita equivalent to $k \mathcal{D}$-mod. Moreover, if $\mathcal{C} = \bigsqcup _{i=1}^{m} \mathcal{C}_i$ is a disjoint union of several full subcategories, then $k\mathcal{C} = k\mathcal{C}_1 \oplus \ldots \oplus k\mathcal{C}_m$ as algebras. Thus it is sufficient to study the representation types of connected, skeletal finite EI categories. We make the following convention:\\

\noindent \textbf{Convention:} \textit{All finite EI categories in this paper are \textbf{connected} and \textbf{skeletal}. Thus endomorphisms, isomorphisms and automorphisms in a finite EI category coincide.}\\

Under this hypothesis, if $x$ and $y$ are two distinct objects in $\mathcal{C}$ with $\mathcal{C} (x,y) \neq \emptyset$, then $\mathcal{C} (y,x)$ is empty. Indeed, if this is not true, we can take $\alpha \in \mathcal{C} (y, x)$ and $\beta \in \mathcal{C} (x,y)$. The composite $\beta \alpha$ is an endomorphism of $y$, hence an automorphism. Similarly, the composite $\alpha \beta$ is an automorphism of $x$. Thus both $\alpha$ and $\beta$ are isomorphisms, so $x$ is isomorphic to $y$. But this is impossible since $\mathcal{C}$ is skeletal and $x \neq y$.

We give here some illustrative examples showing that the representation theory of finite EI categories has applications in relevant fields such as representation theory of finite groups.

\begin{example}
Let $\mathcal{C}$ be the following finite EI category such that: $\mathcal{C} (x, x) \cong \mathcal{C} (y, y) \cong G$ which is a finite group; $\mathcal{C} (x, y) \cong G$ on which $\mathcal{C} (x, x)$ acts by multiplication from the right side and $\mathcal{C} (y, y)$ acts by multiplication from the left side.
\begin{equation*}
\xymatrix{ \mathcal{C}: & x \ar@(ld,lu)[]|{G} \ar @<1ex>[rr] ^{_GG_G} \ar@<-1ex>[rr] ^{\ldots} & & y \ar@(rd,ru)_G}
\end{equation*}
Then a representation $\mathcal{C}$ is nothing but a $kG$-module homomorphism $\varphi: M \rightarrow N$.

Define another finite EI category $\mathcal{D}$ as below, where: $\mathcal{D} (x, x) \cong \mathcal{D} (y, y) \cong \langle g \rangle$ is a cyclic group of order $p = \Char (k)$; $\mathcal{D} (x, y) = \{ \alpha \}$ on which both automorphism groups act trivially.
\begin{equation*}
\xymatrix{ \mathcal{D}: & x \ar@(ld,lu)[]| { \langle g \rangle} \ar[rr] ^{\alpha} & & y \ar@(rd,ru)[]| {\langle g \rangle}}
\end{equation*}
A representation of $\mathcal{D}$ is a $k \langle g \rangle$-module homomorphism $\varphi: M \rightarrow N$ such that $\varphi (\Top (M)) \subseteq \Soc (N)$.
\end{example}

Let $F: \mathcal{D} \rightarrow \mathcal{C}$ be a functor between two finite EI categories. Then $F$ induces a functor Res$_{F}: \mathcal{C}\text{-rep} \rightarrow \mathcal{D}$-rep. Explicitly, for $M \in \mathcal{C}$-rep, Res$_{F} (M) = M \circ F$ is a functor from $\mathcal{D}$ to $k$-vec. This functor Res$_{F}$ is called the \textit{restriction} along $F$ (Definition 3.2.1 in \cite{Xu1}). In the case that $F$ is injective on $\Ob \mathcal{D}$, it induces an algebra homomorphism $\varphi_F: k\mathcal{D} \rightarrow k\mathcal{C}$ by sending $\alpha \in \Mor \mathcal{D}$ to $F(\alpha) \in \Mor \mathcal{C}$ (Proposition 3.1, \cite{Webb2}). In that case, the restriction functor Res$_{F}$ is precisely the usual restriction functor from $k\mathcal{C}$-mod to $k\mathcal{D}$-mod induced  by the algebra homomorphism $\varphi_F$.

An important case is that $\mathcal{D}$ is a subcategory of $\mathcal{C}$ and $F$ is the inclusion functor. Then $k\mathcal{D}$ is a subalgebra of $k\mathcal{C}$ and $\varphi_F$ is the inclusion. In this case we define the \textit{induction functor} $\uparrow _{\mathcal{D}} ^{\mathcal{C}}$ mapping $N \in k\mathcal{D}$-mod to $k\mathcal{C} \otimes _{k \mathcal{D}} N \in k\mathcal{C}$-mod. The \textit{co-induction functor} $\Uparrow _{\mathcal{D}} ^{\mathcal{C}}$ sends $N$ to $N \Uparrow _{\mathcal{D}} ^{\mathcal{C}} = \Hom _{k\mathcal{C}} (k\mathcal{C}, N)$.

In his thesis Xu has developed an induction-restriction theory for finite EI categories. We do not repeat his results here and suggest the reader to look at \cite{Xu1, Xu2} for a detailed description. Although many results in the representation theory of finite groups generalize to finite EI categories, there exist differences, as shown by the following example.

\begin{example}
Let $\mathcal{C}$ be the following EI category such that $g$ has order 2 and permutes $\alpha$ and $\beta$. The category algebra $k\mathcal{C}$ has finite representation type.
\begin{equation*}
\xymatrix{ \mathcal{C}: & x \ar@(ld,lu)[]|{1} \ar @<0.5ex>[rr] ^{\alpha} \ar@<-0.5ex>[rr] _{\beta} & & y \ar@(rd,ru) []|{ \langle g \rangle}}
\end{equation*}
Let $\mathcal{D}$ be the subcategory formed by deleting the morphism $g$. Actually, $\mathcal{D}$ viewed as a quiver is the Kronecker quiver. Therefore, $k\mathcal{D}$ is of infinite representation type. Suppose that $\Char(k) = 0$. Then $R: \xymatrix{ k \ar @<0.5ex>[r] ^1 \ar@<-0.5ex>[r] _3 & k}$ is a representation of $\mathcal{D}$. Clearly, $\dim _k R = 2$. Let $\{u, v\}$ be a basis of $R$ such that $\alpha \cdot u = v$ and $\beta \cdot u = 3v$.

Consider the induced module $\tilde{R}: k\mathcal{C} \otimes _{k\mathcal{D}} R$. We have
\begin{equation*}
g \otimes v = g \otimes \alpha \cdot u = g\alpha \otimes u = \beta \otimes u = 1_y \otimes \beta \cdot u = 1_y \otimes 3v = 3(1_y \otimes v).
\end{equation*}
Therefore,
\begin{equation*}
1_y \otimes v = g^2 \otimes v = g(g \otimes v) = g (3(1_y \otimes v)) = 3 (g \otimes v) = 9(1_y \otimes v).
\end{equation*}
Since $\Char(k) = 0$, we get $1_y \otimes v = 0$. Consequently,
\begin{equation*}
\alpha \otimes u = 1_y \otimes \alpha \cdot u = 1_y \otimes v =0
\end{equation*}
and similarly $\beta \otimes u = 0$. Therefore, $\tilde{R}$ is spanned by $1_x \otimes u$, so is one-dimensional! Actually, it has the form $\tilde{R} (x) \cong k \rightarrow \tilde{R} (y) = 0$.
\end{example}

The following result relates the representation type of $k\mathcal{C}$ to that of the category algebra of a subcategory.

\begin{proposition}
Let $\mathcal{D}$ be a subcategory of a finite EI category $\mathcal{C}$.
\begin{enumerate}
\item If $M \mid M \downarrow _{\mathcal{D}} ^{\mathcal{C}} \uparrow _{\mathcal{D}} ^{\mathcal{C}}$ for all $M \in k\mathcal{C}$-mod, then $k\mathcal{C}$ is of finite representation type whenever so is $k\mathcal{D}$.
\item Dually, if $N \mid N \uparrow _{\mathcal{D}} ^{\mathcal{C}} \downarrow _{\mathcal{D}} ^{\mathcal{C}}$ for all $N \in k\mathcal{D}$-mod, then $k\mathcal{C}$ is of infinite representation type whenever so is $k\mathcal{D}$.
\item If $k\mathcal{D} \mid (_{k\mathcal{D}} k\mathcal{C} _{k\mathcal{D}})$, then $N \mid N \uparrow _{\mathcal{D}} ^{\mathcal{C}} \downarrow _{\mathcal{D}} ^{\mathcal{C}}$ for every $N \in k\mathcal{D}$-mod. In particular, if $k\mathcal{D}$ is of infinite representation type, so is $k\mathcal{C}$.
\end{enumerate}
\end{proposition}

\begin{proof}
Note that $k\mathcal{D}$ is a subalgebra of $k\mathcal{C}$. The conclusions follows from well known results of induction and restriction procedure.
\end{proof}

If furthermore $\mathcal{D}$ is a full subcategory of $\mathcal{C}$, then $k \mathcal{D} = 1_{\mathcal{D}} k\mathcal{C} 1_{\mathcal{D}}$, and the restriction functor Res$_F$ sends $M \in k\mathcal{C}$-mod to $M \downarrow _{\mathcal{D}} ^{\mathcal{C}} = 1_{\mathcal{D}} M$. It has a left adjoint functor $\uparrow _{\mathcal{D}} ^{\mathcal{C}}$ and a right adjoint functor $\Uparrow _{\mathcal{D}} ^{\mathcal{C}}$. In this case, for every $k\mathcal{D}$-module $N$, we have $N \uparrow _{\mathcal{D}} ^{\mathcal{C}} \downarrow _{\mathcal{D}} ^{\mathcal{C}} \cong N$, see \cite{Xu1}. This immediately implies:

\begin{proposition}
Let $\mathcal{D}$ be a full subcategory of a finite EI category $\mathcal{C}$. If $k\mathcal{D}$ is of infinite representation type, so is $k\mathcal{C}$.
\end{proposition}

Now we consider another type of important functors $G: \mathcal{C} \rightarrow \mathcal{D}$ between two finite EI categories, called \textit{quotient functors}. That is, $G$ is a full functor and is bijective restricted to objects. Correspondingly, we call $\mathcal{D}$ a \textit{quotient category} of $\mathcal{C}$. By Proposition 3.1 in \cite{Webb2}, $G$ induces an algebra homomorphism $\varphi_G: k\mathcal{C} \rightarrow k\mathcal{D}$. Moreover, this homomorphism is surjective, so $k\mathcal{D}$ is a quotient algebra of $k\mathcal{C}$, and hence its infinite representation type implies the infinite representation type of $k\mathcal{C}$. Namely we proved

\begin{proposition}
Let $\mathcal{D}$ be a quotient category of a finite EI category $\mathcal{C}$. If $k\mathcal{C}$ is of finite representation type, so is $k\mathcal{D}$.
\end{proposition}

Applying these preliminary results, we obtain the following elementary criteria for finite representation type. Recall the underlying poset $\mathcal{P}$ of a finite EI category $\mathcal{C}$ is defined as follows: elements in $\mathcal{P}$ coincide with objects in $\mathcal{C}$; for two objects $x, y$, we define $x \leqslant y$ if $\mathcal{C} (x, y) \neq \emptyset$. The reader can check that $\mathcal{P}$ is indeed a poset when $\mathcal{C}$ is skeletal.

\begin{proposition}
Let $\mathcal{C}$ be a finite EI category such that $k\mathcal{C}$ is of finite representation type. Then the following conditions must be true:
\begin{enumerate}
\item The group $\mathcal{C} (x, x)$ is of finite representation type for each $x \in \Ob \mathcal{C}$.
\item The underlying poset $\mathcal{P}$ of $\mathcal{C}$ is of finite representation type.
\item For $x \neq y \in \Ob \mathcal{C}$,  $\mathcal{C}(x,y)$ has at most one orbit as a $(\mathcal{C}(y, y), \mathcal{C} (x, x))$-biset.
\item If $x, y, z \in \Ob (\mathcal{C})$ are distinct such that $\mathcal{C} (x,y)$ and $\mathcal{C} (y,z)$ are non-empty, then $\mathcal{C} (x,z) = \mathcal{C} (y,z) \circ \mathcal{C} (x,y) = \{ \beta \circ \alpha \mid \alpha \in \mathcal{C}(x,y), \beta \in \mathcal{C} (y,z) \}$.
\end{enumerate}
\end{proposition}

\begin{proof}
(1): Apply Proposition 2.4 to the full subcategory formed by $x$.

(2): There is a quotient functor $F: \mathcal{C} \rightarrow \mathcal{P}$. It is the identity map restricted to $\Ob \mathcal{C}$, and sends every morphism (if exists) in $\mathcal{C} (x, y)$ to the unique morphism $\mathcal{P} (x, y)$ for $x, y \in \Ob \mathcal{C}$. The conclusion follows from Proposition 2.5.

(3): Consider the full subcategory $\mathcal{D}$ with objects $x$ and $y$. Let $O_1, O_2, \ldots, O_n$ be the $(\mathcal{D}(y, y), \mathcal{D} (x, x))$-orbits of $\mathcal{D} (x, y) = \mathcal{C} (x, y)$. Take a representative $\alpha_i$ for each orbit $O_i$, $1 \leqslant i \leqslant n$. Let $\mathcal{E}$ be the following category: $\Ob \mathcal{E} = \{x, y\}$ and $\Mor {\mathcal{E}} = \{1_x, 1_y, \alpha_1, \ldots, \alpha_n\}$. We claim that $\mathcal{E}$ is a quotient category of $\mathcal{D}$.

Indeed, define $G: \mathcal{D} \rightarrow \mathcal{E}$ as follows. Restricted to $\Ob \mathcal{D}$, $G$ is the identity map; $G (\delta) = 1_x$ for $\delta \in \mathcal{D} (x, x)$; $G (\rho) = 1_y$ for $\rho \in \mathcal{D} (y, y)$; and $G (\alpha) = \alpha_i$ for $\alpha \in O_i \subseteq \mathcal{D} (x, y)$, $1 \leqslant i \leqslant n$. We check that $G$ is well defined. Therefore $\mathcal{E}$ is a quotient category of $\mathcal{D}$, so $k\mathcal{E}$ must be of finite representation type. But this happens if and only if $n \leqslant 1$.

(4): Clearly, $\mathcal{C} (y,z) \circ \mathcal{C} (x,y) \subseteq \mathcal{C} (x, z)$. We show the other inclusion by contradiction. Suppose that $\mathcal{C} (y,z) \circ \mathcal{C} (x,y)$ is a proper subset of $\mathcal{C} (x, z)$. We can take $\varphi$, $\psi \in \mathcal{C} (x,z)$ such that $\varphi \in \mathcal{C} (y,z) \circ \mathcal{C} (x,y)$ and $\psi \notin \mathcal{C} (y,z) \circ \mathcal{C} (x,y)$. Then the $(\mathcal{C} (z, z), \mathcal{C} (x, x))$-orbit generated by $\varphi$ is contained in $\mathcal{C} (y,z) \circ \mathcal{C} (x,y)$. So $\varphi$ and $\psi$ are in different orbits. Therefore, $\mathcal{C} (x, z)$ as a $(\mathcal{C} (z, z), \mathcal{C} (x, x))$-biset has more than one orbits. By (3), $k\mathcal{C}$ is of infinite representation type. The conclusion follows from contradiction.
\end{proof}

In the rest of this paper we only consider finite EI categories $\mathcal{C}$ satisfying all conditions in Proposition 2.6 since otherwise we can immediately conclude that $k\mathcal{C}$ is of infinite representation type.

\section{Criteria of finite representation type}

Throughout this section let $\mathcal{C}$ be a (connected, skeletal) finite EI category with two objects $x$ and $y$. Without loss of generality we assume that $\mathcal{C} (x, y) \neq \emptyset$, but $\mathcal{C} (y, x) = \emptyset$. Conditions (2) and (4) in Proposition 2.6 are satisfied trivially. We suppose that conditions (1) and (3) also hold. That is, both $\mathcal{C} (x, x)$ and $\mathcal{C} (y, y)$ have cyclic Sylow $p$-subgroups, and $\mathcal{C} (x, y)$ has only one orbit as a biset.

Let $G = \mathcal{C} (x, x)$ and $H = \mathcal{C} (y, y)$. Fix a representative morphism $\alpha: x \rightarrow y$ in the unique orbit. Then the structure of $\mathcal{C}$ can be pictured as below:
\begin{equation*}
\xymatrix{ \mathcal{C}: & x \ar@(ld,lu)[]|{G} \ar @<1ex>[rr] ^{H \alpha G} \ar@<-1ex>[rr] ^{\ldots} & & y \ar@(rd,ru)[]|{H}}
\end{equation*}

The following notation will be used in this section: $G_0 = \Stab _{G} (\alpha)$, $H_0 = \Stab _{H} (\alpha)$, $G_1 = \Stab _{G} (H\alpha)$, and $H_1 = \Stab _H (\alpha G)$. By Lemma 3.3 in \cite{Li1}, $G_0 \lhd G_1 \leqslant G$, $H_0 \lhd H_1 \leqslant H$, and $G_1/G_0 \cong H_1/H_0$. We identify these two quotient groups. Clearly, $G$ acts transitively on $\mathcal{C} (x, y)$ if and only if $H_1 = H$, and $H$ acts transitively on $\mathcal{C} (x, y)$ if and only if $G_1 = G$.

\subsection{A description of representations}

Let $R$ be a representation of $\mathcal{C}$. It is easy to check that $R(x)$ (resp., $R(y)$) is a $kG$-module (resp., a $kH$-module). Moreover, suppose that $(\varphi_x, \varphi_y): R_1 \rightarrow R_2$ is a $k\mathcal{C}$-module homomorphism, then $\varphi_x: R_1(x) \rightarrow R_2 (x)$ is a $kG$-module homomorphism, and $\varphi_y: R_1(y) \rightarrow R_2 (y)$ is a $kH$-module homomorphism. For an arbitrary $\beta \in \mathcal{C} (x, y)$, since $\mathcal{C} (x, y)$ has only one orbit as a biset, we can find some $g \in G$ and $h \in H$ such that $\beta = h \alpha g$. Therefore, $R(\beta) = R(h) R(\alpha) R(g)$.

From the above analysis we know that $R$ is uniquely determined by a linear map $R(\alpha)$ from a $kG$-module $V$ to a $kH$-module $W$. Clearly, not every linear map $\varphi: V \rightarrow W$ can be used to define a representation of $\mathcal{C}$. We characterizes in the next proposition those linear maps which indeed define representations of $\mathcal{C}$. Let $V^{\bot}$ be the kernel of $\varphi$, $\bar{V} = V / V^{\bot}$, and $W^{\top}$ be the image of $\varphi$. The map $\varphi$ gives rise to a vector space isomorphism $\bar{\varphi}: \bar{V} \rightarrow W^{\top}$ defined by $\bar{\varphi} (\bar{v}) = \varphi (v)$.

\begin{proposition}
Notation as above. The linear map $\varphi: V \rightarrow W$ determines a representation $R$ of $\mathcal{C}$ by defining $R(\alpha) = \varphi$ if and only if $\bar{V}$ (resp., $W^{\top}$) is a $kG_1$-module (resp., a $kH_1$-module) on which $G_0$ (resp., $H_0$) acts trivially, and $\bar{\varphi}$ is a $k(G_1/G_0) \cong k(H_1/H_0)$-module isomorphism.
\end{proposition}

\begin{proof}
Suppose that there exists a representation $R$ of $\mathcal{C}$ such that $R(\alpha) = \varphi$, i.e, $\varphi(v) = \alpha v$ for all $v \in V$. Take $w \in W^{\top}$. Then we can find some $v \in V$ such that $\alpha v = w$. For every $h \in H_0$, we get $h w = h \alpha v = \alpha v = w$. Thus $H_0$ acts trivially on $W^{\top}$. Moreover, for $h \in H_1$, we can find some $g \in G$ such that $h \alpha = \alpha g$. Therefore, $h w = h \alpha v = \alpha gv = \varphi (gv)$, so $h w \in W^{\top}$. Consequently, $W^{\top}$ is a $kH_1$-module.

Now consider $\bar{V}$. First, for every $g \in G_1$ and $v \in V^{\bot}$, we can find some $h \in H$ such that $h \alpha = \alpha g$. Therefore, $\alpha (g v) = (h \alpha) v = h(\alpha v) = h \varphi(v) = 0$ since $v$ is in the kernel of $\varphi$. We deduce that $V^{\bot}$ is a $kG_1$-module, so $\bar{V}$ is a $kG_1$-module as well. For every $g \in G_0$ and $\bar{v} \in \bar{V}$, we have $\alpha (gv - v) = \alpha gv - \alpha v = \alpha v - \alpha v =0$. Therefore, $gv - v \in V^{\bot}$, and $g \bar{v} - \bar{v} = 0$. Thus $G_0$ acts trivially on $\bar{V}$.

Recall that the isomorphism $\rho: G_1 / G_0 \rightarrow H_1 /H_0$ is defined in the following way (see the proof of Lemma 3.3 in \cite{Li1}): For a given $g \in G_1$, we can find some $h \in H_1$ such that $\alpha g = h\alpha$. Then $\rho$ sends $\bar{g}$ to $\bar{h}$. Now the reader can check that this isomorphism gives an isomorphism between $\bar{V}$ and $W^{\top}$ as $k(G_1 /G_0) \cong k(H_1 / H_0)$-modules. This proves the only if part.

Conversely, if $\varphi$ satisfies the given conditions, then by defining $R(\alpha) (v) = \bar{\varphi} (\bar{v})$ we indeed obtain a representation $R$ of $\mathcal{C}$. To prove this, it suffices to show that if $h_1 \alpha g_1 = h_2 \alpha g_2$ for $g_1, g_2 \in G$ and $h_1, h_2 \in H$, then $h_1 \bar{\varphi} (\overline {g_1 v}) = h_2 \bar{\varphi} ( \overline {g_2 v})$ for every $v \in V$. That is, $h_2^{-1} h_1 \bar{\varphi} (\overline {g_1 v}) = \bar{\varphi} ( \overline {g_2 v})$, where $\overline{g_1 v}, \overline{g_2v}$ are the images of $g_1v$ and $g_2 v$ in $\bar{V}$ respectively.

Note that $h_1 \alpha g_1 = h_2 \alpha g_2$ implies $h_2^{-1}h_1 \alpha = \alpha g_2 g_1^{-1}$, so $h_2^{-1}h_1$ and $g_2 g_1^{-1}$ are contained in $H_1$ and $G_1$ respectively. Passing to the quotient groups, the isomorphism $\rho$ sends $\overline {g_2 g_1^{-1}} \in G_1/G_0$ to $\overline {h_2^{-1} h_1} \in H_1 /H_0$. Since $\bar{V}$ is a $kG_1$-module on which $G_0$ acts trivially and $W^{\top}$ is a $kH_1$-module on which $H_0$ acts trivially, and $\bar{\varphi}$ is a $k(G_1 /G_0) \cong k(H_1 / H_0)$-module isomorphism, we get
\begin{equation*}
h_2^{-1} h_1 \bar{\varphi} (\overline {g_1 v}) = \overline{ h_2^{-1} h_1} \bar{\varphi} (\overline {g_1 v}) = \bar{\varphi} (\overline {g_2 g_1^{-1}} \cdot \overline{g_1v}).
\end{equation*}

Note that in the above identity $\overline{g_1 v}$ cannot be expresses as $g_1 \overline{v}$ since $\bar{V}$ is only a $kG_1$-module instead of a $kG$-module, and the action of $g_1 \in G$ on $\bar{v}$ in general is not defined. However, we have
\begin{equation*}
\overline {g_1 v} = \overline {g_1 g_2^{-1} g_2 v} = (g_1g_2^{-1}) \cdot \overline {g_2v} = \overline{g_1g_2^{-1}} \cdot \overline {g_2v}
\end{equation*}
since $\bar{V}$ is a $kG_1$-module, $g_1g_2^{-1} \in G_1$, and $G_0$ acts trivially on $\bar{V}$. Combining these two identities together, we get
\begin{equation*}
h_2^{-1} h_1 \bar{\varphi} (\overline {g_1 v}) = \bar{\varphi} (\overline {g_2 g_1^{-1}} \cdot \overline{g_1v}) = \bar{\varphi} (\overline {g_2 g_1^{-1}g_1g_2^{-1}} \cdot \overline {g_2v} ) = \bar{\varphi} ( \overline {g_2 v})
\end{equation*}
as required. This finishes the proof.
\end{proof}

\begin{example}
Let $\mathcal{C}$ be a finite EI category with objects $x$ and $y$; $H = \mathcal{C} (y,y)$ is a copy of the symmetric group $S_3$ on 3 letters; $G= \mathcal{C} (x,x)$ is cyclic of order 2; $\mathcal{C} (x,y)= S_3$ regarded as an $(H,G)$-biset where $H$ acts from the left by multiplication, $G$ acts from the right by multiplication after identifying $G$ with a subgroup $G^{\dagger}$ of $S_3$; $\mathcal{C} (y,x) = \emptyset$. We have $G_0 = 1$, $G_1 = G$, $H_0 = 1$, $H_1 \cong C_2$. We set $\Char(k) = 3$ instead of 0 as before.

Let $V = kG$ and $W$ be the projective $kH$-module with $W / \rad W \cong k$. Then
\begin{equation*}
 V= V \downarrow _{G_1}^G = k \uparrow_{G_0}^{G_1} \cong k \oplus \epsilon, \qquad W \downarrow _{H_1} ^H \cong k \oplus k \oplus \epsilon, \qquad k\uparrow _{H_0}^{H_1} \cong k \oplus \epsilon.
\end{equation*}
Therefore, any linear map $\varphi$ with the following matrix representation can define a representation of $\mathcal{C}$: $\begin{bmatrix} \lambda & 0 \\ \mu & 0 \\ 0 & \nu \end{bmatrix}$, $\lambda, \mu, \nu \in k$.
\end{example}

\subsection{Biset decomposition, induction and restriction.}

In Section 2 we have defined the restriction functor and induction functor. Now we use these functors to study some special subcategories of $\mathcal{C}$. The main technique we use here is biset decomposition (see \cite{Bouc}).

The first special case we consider is the following category $\mathcal{D}$: $\Ob \mathcal{D} = \Ob \mathcal{C}$, $\mathcal{D} (x, x) = G_1$, $\mathcal{D} (y, y) = H$, and $\mathcal{D} (x, y) = H \alpha G_1 = H \alpha$. We picture the structure of $\mathcal{D}$ below:
\begin{equation*}
\xymatrix{ \mathcal{D}: & x \ar@(ld,lu)[]|{G_1} \ar @<1ex>[rr] ^{H \alpha G_1} \ar@<-1ex>[rr] ^{\ldots} & & y \ar@(rd,ru)[]|{H}}
\end{equation*}

\begin{lemma}
Let $\mathcal{D}$ be defined as above. Then $k\mathcal{D} \mid (_{k\mathcal{D}} k\mathcal{C} _{k\mathcal{D}})$. In particular, $k\mathcal{C}$ is of infinite representation type if so is $k\mathcal{D}$.
\end{lemma}

\begin{proof}
Let $g_1=1, \ldots ,g_m$ be a chosen set of representatives of the right cosets $G_1 \backslash G$. Clearly,
\begin{equation*}
\Mor \mathcal{C} = G \sqcup H \sqcup H \alpha G = (\bigsqcup _{i=1}^m G_1g_i) \sqcup H \sqcup (\bigcup _{i = 1}^m H \alpha g_i).
\end{equation*}
We claim that $\cup _{i = 1}^m H \alpha g_i$ is a disjoint union. If this is not true, we can find some $1 \leqslant i \neq j \leqslant m$ such that $H \alpha g_i \cap H\alpha g_j \neq \emptyset$. In particular, there exist $h_1, h_2 \in H$ such that $h_1 \alpha g_i = h_2 \alpha g_j$. Thus $h_2^{-1} h_1 \alpha = \alpha g_j g_i^{-1}$, and $g_j g_i^{-1} \in G_1$. Consequently, $g_i$ and $g_j$ are in the same coset. This is a contradiction.

Define $R_1 = k(G_1 \cup H\alpha), R_2 = k(G_1g_2 \cup H\alpha g_2), \ldots, R_m = k(G_1g_m \cup H\alpha g_m)$. They are all left $k\mathcal{D}$-modules. By the above biset decomposition, we get:
\begin{equation*}
_{k\mathcal{D}} k\mathcal{C} = kH \oplus R_1 \oplus R_2 \oplus \ldots \oplus R_m = k\mathcal{D} \oplus R
\end{equation*}
since $R_1 \oplus kH = k\mathcal{D}$.

The above decomposition $_{k\mathcal{D}} k\mathcal{C} = k\mathcal{D} \oplus R$ is also a decomposition of $k\mathcal{C} _{k\mathcal{D}}$. Indeed, $R = (\bigoplus _{i=2}^{m} kG_1g_i) \oplus (\bigoplus _{i=2}^{m} kH\alpha g_i)$ as a $k$-vector space. But both $\bigoplus _{i=2}^{m} kG_1g_i$ and $\bigoplus _{i=2}^{m} kH\alpha g_i$ are right $k\mathcal{D}$-modules. So $R$ is a right $k\mathcal{D}$-module as well. This proves the first statement. The second one follows from Proposition 2.3.
\end{proof}

Now consider the second special case. We introduce some notation here. Let $K \leqslant G_1$ be a subgroup containing $G_0$. Then $\alpha K \subseteq \alpha G_1 \subseteq H \alpha$. Define $L = \Stab _H (\alpha K)$. That is, $L = \{ h \in H \mid h \alpha K = \alpha K \} = \{ h \in H \mid \exists g \in K \text{ such that } h\alpha = \alpha g\}$.

\begin{lemma}
Notation as above. Then $L$ is a subgroup of $H_1$, and $L/H_0 \cong K/G_0$.
\end{lemma}

\begin{proof}
We can check that $L$ is a subgroup of $H_1$ as we did in the proof of Lemma 3.3 in \cite{Li1}. Since $H_0 \lhd H_1$ by that lemma, we know $H_0 \lhd L$. The fact that $L / H_0 \cong K / G_0$ can be shown as in Lemma 3.3 of \cite{Li1}, too.
\end{proof}

Thus $H_0 \lhd L \leqslant H_1 \leqslant H$ and $G_0 \lhd K \leqslant G_1 \leqslant G$. We then construct a subcategory $\mathcal{E}$ as follows:
\begin{equation*}
\xymatrix{ \mathcal{E}: & x \ar@(ld,lu)[]|{K} \ar @<1ex>[rr] ^{L\alpha K} \ar@<-1ex>[rr] ^{\ldots} & & y \ar@(rd,ru)[]|{L}}
\end{equation*}
The reader can check that both $L$ and $K$ act transitively on $\mathcal{E} (x, y)$ since $L \alpha K = L \alpha = \alpha K$ by the definitions of $K$ and $L$.

\begin{lemma}
Let $\mathcal{E}$ be constructed as above and suppose that both $G$ and $H$ act transitively on $\mathcal{C} (x, y)$. Then $k\mathcal{E} \mid (_{k\mathcal{E}} k\mathcal{C} _{k\mathcal{E}})$. In particular, $k\mathcal{C}$ is of infinite representation type if so is $k\mathcal{E}$.
\end{lemma}

\begin{proof}
Choose a representative from each coset in $K \backslash G$: $g_1 = 1, \ldots, g_m$. Since both $G$ and $H$ acts transitively on $\mathcal{C} (x, y)$, we know that $G= G_1$ and $H = H_1$. Moreover, $G_1/G_0 \cong H_1 / H_0$. For each $g_i$, $1 \leqslant i \leqslant m$, there is some $h_i$ such that $\alpha g_i = h_i \alpha$. We claim that $Lh_i \alpha \cap Lh_j \alpha = \emptyset$, and hence $L h_i \cap L h_j = \emptyset$, $1 \leqslant i \neq j \leqslant m$. Indeed, since $L h_i \alpha = L \alpha g_i = \alpha K g_i$ and $L h_j \alpha = L \alpha g_j = \alpha K g_j$. If the intersection is not empty, there exist $u, v \in K$ such that $\alpha u g_i = \alpha v g_j$, i.e, $\alpha = \alpha v g_j g_i^{-1} u^{-1}$. Therefore, $v g_j g_i^{-1} u^{-1} \in G_0 \leqslant K$, and hence $g_j g_i^{-1} \in K$. This is impossible since $g_i$ and $g_j$ are in different cosets.

Note that $\{ h_i \} _{i=1} ^m$ is a set of representatives of the cosets $L \backslash H$ since $G / G_0 \cong H /H_0$ and $K / G_0 \cong L/ H_0$. Clearly,
\begin{equation*}
\Mor \mathcal{C} = G \sqcup H \sqcup H \alpha = (\bigsqcup _{i=1}^m Kg_i) \sqcup (\bigsqcup _{j=1}^m Lh_j) \sqcup (\bigcup _{j = 1}^m Lh_j \alpha).
\end{equation*}
The vector space spanned by $K \sqcup L\alpha \sqcup L$ is precisely $k\mathcal{E}$. Let $R$ be the vector space spanned by $(\sqcup _{i=2}^m Kg_i) \sqcup (\sqcup _{j=2}^m Lh_j) \sqcup (\sqcup _{j=2}^m Lh_j \alpha)$. Then $R$ is a left $k\mathcal{E}$-module, and hence $_{k\mathcal{E}} k\mathcal{C} = k\mathcal{E} \oplus R$. We also have $k\mathcal{C} _{k\mathcal{E}} = k\mathcal{E} \oplus R$. Indeed, since
\begin{equation*}
\sqcup _{j=2}^m Lh_j \alpha K = \sqcup _{j=2}^m Lh_j L \alpha = \sqcup _{j=2}^m L h_j \alpha,
\end{equation*}
we check that $R$ is a right $k\mathcal{E}$-module as well. This proves the first statement. The second one follows from Proposition 2.3.
\end{proof}

Given a subcategory $\mathcal{D}$ of $\mathcal{C}$ as pictured below and $N \in k\mathcal{D}$-mod, where $G' = \mathcal{D} (x, x)$ and $H' = \mathcal{D} (y, y)$.
\begin{equation*}
\xymatrix{ \mathcal{D}: & x \ar@(ld,lu) []|{G'} \ar @<1ex>[rr] ^{H' \alpha G'} \ar@<-1ex>[rr] ^{\ldots} & & y \ar@(rd,ru) []|{H'}}.
\end{equation*}
We want to describe the structure of the induced module $N \uparrow _{\mathcal{D}} ^{\mathcal{C}}$. As described in Proposition 3.1, $N$ is determined by a linear map $\varphi: V \rightarrow W$ where $V = N(x)$ and $W = N(y)$ are a $kG'$-module and a $kH'$-module respectively. In general, $N \uparrow _{\mathcal{D}} ^{\mathcal{C}} (x) \ncong V \uparrow _{G'}^G$ and $N \uparrow _{\mathcal{D}} ^{\mathcal{C}} (y) \ncong W \uparrow _{H'}^H$. Indeed, we have already seen in Example 2.2 that the dimension of the induced module can be even smaller than that of the original module. In the following example we see that the induced module is isomorphic to the original one.

\begin{example}
Let $\mathcal{C}$ be a finite EI category such that: $\Ob \mathcal{C} = \{ x, y\}$, $G = 1$, $H = \langle h \rangle$ is a cyclic group of order $2 = \Char (k)$, and $\mathcal{C} (x, y) = \{ \alpha \}$. Let $\mathcal{D}$ be the subcategory formed by removing $h \in H$. Let $\varphi: V \rightarrow W$ be a representation $R$ of $\mathcal{D}$, where $V = \langle v \rangle \cong k$, $W = \langle w \rangle \cong k$ and $\varphi (v) = w$. By direct computation we get $R \uparrow _{\mathcal{D}} ^{\mathcal{C}} (x) \cong V$ and $R \uparrow _{\mathcal{D}} ^{\mathcal{C}} (y) \cong W \ncong W \uparrow _{1}^H$:
\begin{equation*}
h \otimes _{k \mathcal{D}} w = h \otimes _{k\mathcal{D}} \alpha \cdot v = h \alpha \otimes _{k\mathcal{D}} v = \alpha \otimes _{k\mathcal{D}} v = 1_y \otimes _{k\mathcal{D}} \alpha \cdot v = 1_y \otimes w.
\end{equation*}
\end{example}

If the pair $(\mathcal{C}, \mathcal{D})$ satisfies some extra property, we can explicitly describe the structure of the induced module $R \uparrow _{\mathcal{D}} ^{\mathcal{C}}$ by a linear map between two induced group representations.

\begin{lemma}
Suppose that $H$ acts transitively on $\mathcal{C} (x, y)$. Let $\mathcal{D}$ be a subcategory of $\mathcal{C}$ pictured as above and $R$ be a $k\mathcal{D}$-module of the form $R(\alpha) = \varphi: V \rightarrow W$. If $\Stab _H (\alpha G') \leqslant H'$, then $\tilde{R} = R \uparrow _{\mathcal{D}} ^{\mathcal{C}}$ has the form
\begin{equation*}
\tilde{R} (\alpha) = \tilde{\varphi}: kG \otimes _{kG'} V \rightarrow kH \otimes _{kH'} W, \quad g \otimes v \mapsto h \otimes \varphi (v),
\end{equation*}
where $h \in H$ such that $\alpha g = h\alpha$.
\end{lemma}

\begin{proof}
We first check that $\tilde{\varphi}$ indeed determines a $k\mathcal{C}$-module. Since $kG \otimes _{kG'} V$ and $kH \otimes _{kH'} W$ have a natural $kG$-module structure and a natural $kH$-module structure respectively, it is enough to prove that $h_1 \tilde{\varphi} g_1$ and $h_2 \tilde{\varphi} g_2$ coincide as linear maps from $kG \otimes _{kG'} V$ to $kH \otimes _{kH'} W$ if $h_1 \alpha g_1=h_2 \alpha g_2$, $g_1, g_2 \in G$, $h_1, h_2 \in H$.

Take $g \otimes _{kG'} v \in kG \otimes_{kG'} V$. Since $H$ acts transitively, we can find some $h \in H$ such that $h\alpha = \alpha g$. Moreover, $h_1 \alpha g_1=h_2 \alpha g_2$ implies $\alpha g_1 = h_1^{-1} h_2 \alpha g_2$, so $\alpha g_1 g = h_1^{-1} h_2 \alpha g_2 g$. Therefore,
\begin{equation*}
h_1 \tilde{\varphi} g_1 (g \otimes _{kG'} v) = h_1 \tilde{\varphi} (g_1g \otimes _{kG'} v) = h_1 h \otimes _{kH'} \varphi(v),
\end{equation*}
where $h \alpha = \alpha g_1 g = h_1^{-1} h_2 \alpha g_2 g$.
On the other hand, we also have
\begin{equation*}
h_2 \tilde{\varphi} g_2 (g \otimes _{kG'} v) = h_2 \tilde{\varphi} (g_2g \otimes _{kG'} v) = h_2 h' \otimes _{kH'} \varphi(v),
\end{equation*}
where $h'\alpha = \alpha g_2 g$. Thus $h \alpha = h_1^{-1} h_2 \alpha g_2 g = h_1^{-1} h_2 h' \alpha$, and $\alpha  = h^{-1} h_1^{-1}h_2 h' \alpha$. Consequently, $h^{-1} h_1^{-1}h_2 h' \in H_0 \lhd H' \leqslant H_1$. By Proposition 3.1, $H_0$ acts trivially on $\varphi(v)$, so we have
\begin{equation*}
h_1 h \otimes _{kH'} \varphi(v) = h_1h \otimes _{kH'} h^{-1} h_1^{-1}h_2 h' \varphi(v) = h_2 h' \otimes _{kH'} \varphi(v)
\end{equation*}
as required. Therefore, $\tilde{\varphi}$ is a well defined representation of $\mathcal{C}$.

Now Let $\{v_i\}_{i=1}^{m}$ and $\{w_j\}_{i=1}^{n}$ be bases of $V$ and $W$ respectively. Note that every morphism in $\mathcal{C} (x, y)$ can be written as $h \alpha$ for some $h \in H$. We define a $k$-linear map $\pi: k\mathcal{C}\times R \rightarrow \tilde{R}$ by:
\begin{align*}
(g,v_i) \mapsto g \otimes _{kG'} v_i, \quad (h \alpha, v_i) \mapsto h \otimes _{kH'} \varphi(v_i), \quad (h,w_j) \mapsto h \otimes _{kH'} w_j\\
(h,v_i) \mapsto 0, \quad (\alpha, w_j) \mapsto 0, \quad (g,w_j)\mapsto 0
\end{align*}
for $g\in G, h\in H, 1 \leqslant i \leqslant m, 1 \leqslant j \leqslant n$.
The reader can check that $\pi$ is $k\mathcal{D}$-balanced using either the condition $\Stab _H (\alpha G') \leqslant H'$, so it induces the following $k\mathcal{D}$-homomorphism $\tilde{\pi}: R\uparrow_{\mathcal{D}}^{\mathcal{C}} \rightarrow \tilde{R}$:
\begin{equation*}
g \otimes _{k\mathcal{D}} v_i \mapsto g \otimes _{kG'} v_i, \quad h \otimes _{k\mathcal{D}} w_j \mapsto h \otimes _{kH'} w_j
\end{equation*}
for $g\in G, h\in H, 1 \leqslant i \leqslant m, 1 \leqslant j \leqslant n$, as shown by the following commutative diagram:
\begin{equation*}
\xymatrix{
g \otimes _{k\mathcal{D}} v_i \ar@{|-{>}}[r] \ar@ {|-{>}}[d] ^{\widetilde {\pi_{x}}} & \alpha g \otimes _{k\mathcal{D}} v_i = h \alpha \otimes_ {k\mathcal{D}} v_i = h \otimes _{k\mathcal{D}} \varphi(v_i) \ar@ {|-{>}}[d] ^{\widetilde {\pi_y}} \\
g \otimes _{kG'} v_i \ar@ {|-{>}}[r]_ -{\tilde {\varphi}} & \tilde{\varphi} (g \otimes _{kG'} v_i) = h \otimes _{kH'} \varphi(v_i)
}
\end{equation*}

Note that $\tilde{\pi}$ is not only a $k\mathcal{D}$-module homomorphism, but also a $k\mathcal{C}$-module homomorphism. Furthermore, from our definition of $\pi$ we find that all basis elements of $\tilde{R}$ are images under $\pi$. By the universal property of tensor product, they are actually images under $\tilde{\pi}$. Therefore, $\tilde{\pi}$ is a surjective $k\mathcal{C}$-module homomorphism.

On the other hand, $R \uparrow _{\mathcal{D}} ^{\mathcal{C}} (x)$ is spanned by basic tensors $g \otimes _{k\mathcal{D}} v_i$, $1 \leqslant i \leqslant m$. Since $kG' \subset k\mathcal{D}$, it is actually spanned by the set $\{ g_t \otimes _{k\mathcal{D}} v_i \mid 1 \leqslant t \leqslant l, 1 \leqslant i \leqslant m \}$, where $\{ g_t \} _{t=1}^{l}$ is a set of representatives of the left cosets $G/G'$. It turns out that
\begin{equation*}
\text{dim}_k (R \uparrow _{\mathcal{D}} ^{\mathcal{C}} (x)) \leqslant |G: G'| \text{dim}_k V = \text{dim}_k \tilde{R}(x).
\end{equation*}

Similarly, since $H$ acts transitively on $\mathcal{C} (x,y)$, any morphism $\beta$ in $\mathcal{C} (x, y)$ can be written as $h \alpha$ for some $h \in H$. Thus $\beta (g \otimes _{k\mathcal{D}} v_i) = h \alpha g \otimes _{k\mathcal{D}} v_i = h h' \otimes _{k\mathcal{D}} \varphi(v_i)$ where $\alpha g = h' \alpha$. Therefore, $R \uparrow _{\mathcal{D}} ^{\mathcal{C}} (y)$ is spanned by basic tensors $h \otimes _{k\mathcal{D}} w_j$, $1 \leqslant j \leqslant n, h \in H$. Since $kH' \subset k\mathcal{D}$, it is actually spanned by the set $\{ h_s \otimes _{k\mathcal{D}} w_j \mid 1 \leqslant s \leqslant r, 1 \leqslant j \leqslant n \}$, where $\{ h_s \} _{s=1}^r$ is a set of representatives of the cosets $H/H'$. Therefore,
\begin{equation*}
\text{dim}_k (R \uparrow _{\mathcal{D}} ^{\mathcal{C}} (y)) \leqslant |H: H'| \text{dim}_k W = \text{dim}_k \tilde{R}(y).
\end{equation*}

By the above two inequalities, we have $\dim_k (R \uparrow _{\mathcal{D}} ^{\mathcal{C}}) \leqslant \dim_k \tilde{R}$, so $\tilde{\pi}$ is actually an isomorphism.
\end{proof}

We will use the following result to get a family of finite EI categories whose category algebras are of finite representation type.

\begin{proposition}
Let $\mathcal{D}$ be as in the previous lemma. Suppose that both $|G : G'|$ and $|H : H'|$ are invertible in $k$. If the pair $(\mathcal{C}, \mathcal{D})$ satisfies one of the following conditions:
\begin{enumerate}
\item $G$ acts transitively on $\mathcal{C} (x, y)$;
\item $H' = H$,
\end{enumerate}
then $M$ is isomorphic to a direct summand of $M \downarrow _{\mathcal{D}} ^{\mathcal{C}} \uparrow _{\mathcal{D}} ^{\mathcal{C}}$ for $M \in k\mathcal{C}$-mod. In particular, $k\mathcal{C}$ has finite representation type if so does $k\mathcal{D}$.
\end{proposition}

\begin{proof}
The second statement follows from Proposition 2.3, so we only need to show the first statement. Firstly we consider a special case. That is, $G$ acts transitively on $\mathcal{C} (x, y)$ and $G' = \Stab _G (H' \alpha)$. Observe that by this given condition $G_0 \lhd G_1 = G$, $H_1 \lhd H_1 = H$, and $H / H_0 \cong G / G_0$. Moreover, by Lemma 3.4, $G' / G_0 \cong H' / H_0$. Take $\{ g_i \} _{i=1}^n$ as a set of representatives of the cosets $G / G'$ and choose $h_i \in H$ satisfying $\alpha g_i = h_i \alpha$. Then $\{ h_i \}_{i=1}^n$ is a set of representatives of the cosets $H / H'$, and $\alpha g_i G' = h_iH' \alpha$. Note that $n$ is invertible.

Suppose that $M$ has the form $\varphi: V \rightarrow W$ where: $V = M(x)$ is a $kG$-module, $W = M(y)$ is a $kH$-module, and $\varphi = M(\alpha)$. By the previous lemma, $M \downarrow _{\mathcal{D}} ^{\mathcal{C}} \uparrow _{\mathcal{D}} ^{\mathcal{C}}$ has the form
\begin{equation*}
\tilde {\varphi}: kG \otimes_{kG'} V \rightarrow kH \otimes _{kH'} W, \quad g \otimes v \mapsto h \otimes \varphi (v),
\end{equation*}
where $h \alpha  = \alpha g$.

Define the following maps:
\begin{align*}
\theta_V: kG \otimes _{kG'} V \rightarrow V, \quad g \otimes v \mapsto gv;\\
\delta_V: V \rightarrow kG \otimes _{kG'} V, \quad v \mapsto \frac{1}{n} \sum_{i=1}^n g_i \otimes g_i^{-1} v;\\
\theta_W: kH \otimes _{kH'} W \rightarrow W, \quad h \otimes w \mapsto hw;\\
\delta_W: W \rightarrow kH \otimes _{kH'} W, \quad w \mapsto \frac{1}{n} \sum_{i=1}^n h_i \otimes h_i^{-1} w.
\end{align*}
and consider the diagram
\begin{equation*}
\xymatrix{
kG \otimes _{kG'} V \ar[r] _{\tilde {\varphi}} \ar@<0.5ex>[d] ^{\theta_V} & kH \otimes _{kH'} W \ar@<0.5ex>[d] ^{\theta_W}\\
V \ar[r]_{\varphi} \ar@<0.5ex> [u]^{\delta_V} & W \ar@<0.5ex> [u]^{\delta_W}.}
\end{equation*}
We can show that the diagram commutes. Indeed, for $g \in G$ and $v \in V$,
\begin{align*}
\varphi \theta_V (g \otimes v) & = \varphi (gv) = (\alpha g) \cdot v = (h \alpha) \cdot v\\
& = h \varphi (v) = \theta_W (h \otimes \varphi(v)) = \theta_W \tilde{\varphi} (g \otimes v)
\end{align*}
where $\alpha g = h \alpha$; and
\begin{align*}
\tilde{\varphi} \delta_V (v) & = \frac{1}{n} \sum_{i=1}^n \tilde{\varphi} (g_i \otimes g_i^{-1} v) = \frac{1}{n} \sum_{i=1}^n h_i \otimes \varphi (g_i^{-1} v) = \frac{1}{n} \sum_{i=1}^n h_i \otimes (\alpha g_i^{-1}) \cdot v\\
& = \frac{1}{n} \sum_{i=1}^n h_i \otimes h_i^{-1} \alpha \cdot v = \frac{1}{n} \sum_{i=1}^n h_i \otimes h_i^{-1} \varphi(v) = \delta_W \varphi(v).
\end{align*}
But $\theta_V \delta_V = 1_V$ and $\theta_W \delta_W = 1_W$. Therefore, $M \mid M \downarrow _{\mathcal{D}} ^{\mathcal{C}} \uparrow _{\mathcal{D}} ^{\mathcal{C}}$.\\

Now suppose that $\mathcal{D}$ satisfies the second condition, i.e., $H' = H$. Thus $\mathcal{D} (x, y) = H \alpha = \mathcal{C} (x, y)$. By the previous lemma, $M \downarrow _{\mathcal{D}} ^{\mathcal{C}} \uparrow _{\mathcal{D}} ^{\mathcal{C}}$ has the form
\begin{equation*}
\tilde {\varphi}: kG \otimes_{kG'} V \rightarrow kH \otimes _{kH} W \cong W, \quad g \otimes v \mapsto h \otimes \varphi (v) = 1 \otimes h \varphi(v),
\end{equation*}
where $\alpha g = h \alpha$.

Let $\{ g_i \} _{i=1}^n$ be a set of representatives of the cosets $G / G'$. Define $\theta_V$ and $\delta_V$ as in the first case. We can show the following diagram commutes:
\begin{equation*}
\xymatrix{
kG \otimes _{kG'} V \ar[r] _{\tilde {\varphi}} \ar@<0.5ex>[d] ^{\theta_V} & kH \otimes _{kH} W \ar@<0.5ex>[d] ^{\cong}\\
V \ar[r]_{\varphi} \ar@<0.5ex> [u]^{\delta_V} & W \ar@<0.5ex> [u]^{\cong}.}
\end{equation*}
So the conclusion follows as in the above special case.

In the situation that $\mathcal{D}$ satisfies the first condition, i.e., $G$ acts transitively on $\mathcal{C} (x, y)$, we define another subcategory $\mathcal{E}$ of $\mathcal{C}$ as follows:, where $\hat{G} = \Stab _{G} (H' \alpha)$:
\begin{equation*}
\xymatrix{ \mathcal{E}: & x \ar@(ld,lu)[]| {\hat{G}} \ar @<1ex>[rr] ^{H' \alpha \hat{G}} \ar@<-1ex>[rr] ^{\ldots} & & y \ar@(rd,ru)[]|{H'}}.
\end{equation*}
Note that both $\hat{G}$ and $H'$ act transitively on $\mathcal{E} (x, y)$, so $\Stab_H (\alpha \hat{G}) = H'$.

The pair $(\mathcal{C}, \mathcal{E})$ falls into the special case we consider at the beginning, so we have $M \mid M \downarrow _{\mathcal{E}} ^{\mathcal{C}} \uparrow _{\mathcal{E}} ^{\mathcal{C}}$. The pair $(\mathcal{E}, \mathcal{D})$ satisfies the second condition, so we have $(M \downarrow _{\mathcal{E}} ^{\mathcal{C}}) \mid (M \downarrow _{\mathcal{E}} ^{\mathcal{C}}) \downarrow _{\mathcal{D}} ^{\mathcal{E}} \uparrow _{\mathcal{D}} ^{\mathcal{E}}$. In conclusion, we get $M \mid M \downarrow _{\mathcal{E}} ^{\mathcal{C}} \downarrow _{\mathcal{D}} ^{\mathcal{E}} \uparrow _{\mathcal{D}} ^{\mathcal{E}} \uparrow _{\mathcal{E}} ^{\mathcal{C}}$, i.e., $M \mid M \downarrow _{\mathcal{D}} ^{\mathcal{C}} \uparrow _{\mathcal{D}} ^{\mathcal{C}}$. This finishes the proof.
\end{proof}

\subsection{Transitivity of group actions.}

Our main task in this subsection is to prove the following proposition, which is  condition (2) in Theorem 1.1.

\begin{proposition}
If neither $G$ nor $H$ acts transitively on $\mathcal{C} (x, y)$, then $k\mathcal{C}$ is of infinite representation type.
\end{proposition}

This conclusion has been proved in \cite{Li1} for the case that both $|G|$ and $|H|$ are invertible. We prove the fact for the general case. Define a finite EI category $\mathcal{Q}$ as follows, where $\Stab_G (\bar{\alpha}) = G_1$ and $\Stab _H (\bar {\alpha}) = H_1$.
\begin{equation*}
\xymatrix{ \mathcal{Q}: & x \ar@(ld,lu)[]|{G} \ar @<1ex>[rr] ^{H \bar{\alpha} G} \ar@<-1ex>[rr] ^{\ldots} & & y \ar@(rd,ru)[]|{H}}.
\end{equation*}

We claim that $\mathcal{Q}$ is a quotient category of $\mathcal{C}$. Indeed, let $F: \mathcal{C} \rightarrow \mathcal{Q}$ be the functor such that $F$ is the identity map on objects and automorphisms, and sends $h \alpha g$ to $h \bar{\alpha} g$ for $h \in H$ and $g \in G$. If $h_1 \alpha g_1 = h_2 \alpha g_2$, $h_1, h_2 \in H$, $g_1, g_2 \in G$, then $h_2^{-1} h_1 \alpha = \alpha g_2 g_1^{-1}$, so $h_2^{-1} h_1 \in H_1$ and $g_2 g_1^{-1} \in G_1$. Therefore, $h_2^{-1} h_1 \bar{\alpha} = \bar{\alpha} g_2 g_1^{-1} = \bar{\alpha}$. In other words, $h_1 \bar{\alpha} g_1 = h_2 \bar{\alpha} g_2$. Thus $F$ is a well defined quotient functor, and $\mathcal{Q}$ is a quotient category of $\mathcal{C}$. Therefore, it suffices to show the infinite representation type of $k\mathcal{Q}$. The strategy is to use permutation modules $M = k \uparrow _{G_1}^G$ and $N = k \uparrow _{H_1}^H$ to construct infinitely many distinct indecomposable representations (up to isomorphism) for $\mathcal{Q}$. Since neither $G$ nor $H$ acts transitively on $\mathcal{C} (x, y)$, we know that $G_1 \neq G$ and $H_1 \neq H$, so $\dim _k M >1$ and $\dim_k N > 1$.

Note that $M$  has a unique indecomposable summand $M_0$ which is characterized by one of the following properties:
\begin{enumerate}
\item $\Hom_{kG} (M_0, k) \cong k$, i.e, $\Top(M_0)$ has a unique summand isomorphic to $k$;
\item $\Hom_{kG} (k, M_0) \cong k$, i.e, $\Soc(M_0)$ has a unique summand isomorphic to $k$.
\end{enumerate}
Moreover, $M_0$ is isomorphic to its dual, and $\Top (M_0) \cong \Soc(M_0)$. This module is called a \textit{Scott module}. In one extreme situation that $H_1$ contains a Sylow $p$-subgroup of $H$, $M_0 \cong k$. If $|H_1|$ is invertible, then $M_0 \cong P_k$, the projective cover of $k$. For more details, see \cite{Broue, Green}. In particular, for a finite group with cyclic Sylow $p$-subgroups (as we study in this paper), the structure of all Scott modules has been described in \cite{Takahashi} by considering the Brauer graphs. Similarly, $N = k \uparrow _{H_1}^H$ has a Scott module $N_0$.

\begin{proof}
We prove the conclusion by considering the permutation modules $M$ and $N$. There are three cases:\\

\textbf{Case I:} Both $M$ and $N$ have at least two indecomposable summands. Take an indecomposable summand $M_1$ of $M$ which is not the Scott module, and let $S$ be a simple summand of $\Soc (M_1)$. Clearly, $S \ncong k$. Moreover, since $\Hom_{kG} (S, M_1) \neq 0$, we deduce that $\Hom _{kG_1} (S \downarrow _{G_1}^G, k) \cong \Hom _{kG} (S, M) \neq 0$, so $\Top (S \downarrow _{G_1}^G)$ has a simple summand $k_1 \cong k$. Dually, we can find a simple $kH$-module $T \ncong k$ such that $\Soc (T \downarrow _{H_1}^H)$ has a simple summand $k_2 \cong k$.

As vector spaces, $S = k_1 \oplus S'$ and $T = k_2 \oplus T'$ ($S'$ and $T'$ might be 0, but this is fine). Construct a family of representations $R_d$ of $\mathcal{Q}$ as follows, $0 \neq d \in k$. First, $R_d (x) = k \oplus S = k \oplus (k_1 \oplus S')$, $R_d(y) = k \oplus T = k \oplus (k_2 \oplus T')$. The linear map $R_d(\bar{\alpha})$ is defined by:
\begin{equation*}
\xymatrix {k \oplus k_1 \oplus S' \ar[rrr] ^{\begin{bmatrix} 1 & 1 & 0 \\ 1 & d & 0 \\ 0 & 0 & 0 \end{bmatrix}} & & & k \oplus k_2 \oplus T'}
\end{equation*}
We check that this linear map $R_d (\bar{\alpha})$ indeed gives rise to a representation of $\mathcal{Q}$ by Proposition 3.1, and that $R_d$ is indecomposable by computing its endomorphism algebra. Moreover, if $R_b \cong R_d$, we have the following matrix identity:
\begin{equation*}
\begin{bmatrix} 1 & 1 & 0 \\ 1 & d & 0 \\ 0 & 0 & 0 \end{bmatrix} \begin{bmatrix} u & 0 & 0 \\ 0 & v & 0 \\ 0 & 0 & v \end{bmatrix} = \begin{bmatrix} \lambda & 0 & 0 \\ 0 & t & 0 \\ 0 & 0 & t \end{bmatrix} \begin{bmatrix} 1 & 1 & 0 \\ 1 & b & 0 \\ 0 & 0 & 0 \end{bmatrix},
\end{equation*}
where $b, d, u, t, v, \lambda \in k$ are nonzero scalars. By computation, we get $b = d$. In this way we construct infinitely many pairwise non-isomorphic indecomposable representations of $\mathcal{Q}$ illustrated as follows, where dotted arrows means that these maps are actually only defined on some subspaces.
\begin{equation*}
\xymatrix{k \ar[rr]^1 \ar@{-->}[drr]^1 & & k \\ S \ar@{-->} [rr]^d \ar@{-->} [urr]_1 & & T.}
\end{equation*}

\textbf{Case II:} One of $M$ and $N$ is indecomposable and the other one is not. Without loss of generality we assume that $M$ is indecomposable. Then $M \ncong k$ is the Scott module. Its length is at least 2; and both $\Top (M)$ and $\Soc (M)$ has a unique composition factor isomorphic to $k$. Let $k_1$ be this composition factor in $\Top(M)$.

Observe that $\dim_k \End _{kG} (M) \geqslant 2$ since there is a nilpotent map $\varphi: M \rightarrow M$ such that $\varphi(M) \cong k$ is a simple summand of $\Soc(M)$. We conclude that
\begin{equation*}
\text{dim}_k \Hom _{kG_1} (M \downarrow _{G_1} ^G, k) = \text{dim}_k \Hom _{kG} (M, k \uparrow _{G_1}^G) = \text{dim}_k \Hom _{kG} (M, M) \geqslant 2.
\end{equation*}
This fact also follows from
\begin{align*}
\text{dim} _k \Hom _{kG_1} (M \downarrow _{G_1} ^G, k) & = \text{dim}_k \Hom _{kG_1} (k \uparrow _{G_1}^G \downarrow _{G_1}^G, k)\\
& = \text{dim} _k \Hom _{kG_1} (\bigoplus _{s \in G_1 \backslash G / G_1} (s \otimes k) \uparrow _{G_1 \cap sG_1s^{-1}} ^{G_1}, k).
\end{align*}
The last dimension equals the number of double cosets $G_1 \backslash G /G_1$, which is at least 2. Therefore, there are at least two summands isomorphic to $k$ in $\Top (M \downarrow _{G_1}^G)$. Take a summand $k_2 \cong k$ in $\Top (M \downarrow _{G_1}^G)$ which is different from $k_1$.

Write $M = k_1 \oplus k_2 \oplus M'$ as vector spaces. Take a simple $kH$-module $T$ as we did in Case I. Note that $T \downarrow _{H_1} ^H = k_3 \oplus T'$ where $k_3 \cong k$. We construct a class of representations $R_d$ in the following way, $0 \neq d \in k$. First, $R_d (x) = M$, $R_d(y) = k \oplus T$. The map $R_d(\bar{\alpha})$ is defined by
\begin{equation*}
\xymatrix {k_1 \oplus k_2 \oplus M' \ar[rrr] ^{\begin{bmatrix} 1 & 1 & 0 \\ 1 & d & 0 \\ 0 & 0 & 0 \end{bmatrix}} & & & k \oplus k_3 \oplus T'}.
\end{equation*}
Again, by Proposition 3.1, $R_{d} (\bar{\alpha})$ gives rise to a representation of $\mathcal{Q}$. It is indecomposable since $R(x) = M$ is an indecomposable $kG$-module. Moreover, we check that $R_{d} \cong R_b$ if and only if $d = b$. In this way we construct infinitely many pairwise non-isomorphic indecomposable representations of $\mathcal{Q}$.\\

\textbf{Case III:} Both $M$ and $N$ are indecomposable. As in Case II, there exist at least two summands $k_1 \cong k_2 \cong k$ in $\Top (M \downarrow _{G_1} ^G)$, where $k_1$ is the restriction of the simple summand $k$ in $\Top(M)$ to $G_1$. Dually, there exist at least two summands $k_3 \cong k_4 \cong k$ in $\Soc (N \downarrow _{H_1} ^H)$, where $k_3$ is the restriction of the simple summand $k$ in $\Soc(N)$ to $H_1$.

Write $M = k_1 \oplus k_2 \oplus M_1$ and $N = k_3 \oplus k_4 \oplus N_1$ as vector spaces. Then the following map $\varphi$
\begin{equation*}
\xymatrix {k_1 \oplus k_2 \oplus M_1 \ar[rrr] ^{\begin{bmatrix} 1 & 1 & 0 \\ 1 & d & 0 \\ 0 & 0 & 0 \end{bmatrix}} & & & k_3 \oplus k_4 \oplus N_1}
\end{equation*}
gives rise to a representation $R_d$ of $\mathcal{Q}$ by Proposition 3.1. It is obviously indecomposable. Moreover, letting $d \neq 0$ vary in $k$, we check similarly that $R_d \cong R_b$ if and only if $d = b$.
\end{proof}

This proposition implies Theorem 3.35 in \cite{Dietrich1}, which asserts that if both $H$ and $G$ are nontrivial and $\mathcal{C} (x, y) \cong H \times G$ as a $(H, G)$-biset, then $k\mathcal{C}$ is of infinite representation type.

\subsection{Actions of $p$-subgroups.}

In this subsection we set $\Char(k) = p \geqslant 5$. It is well known that the representation type of a finite group is completely determined by its Sylow $p$-subgroups. For the finite EI category $\mathcal{C}$, the Sylow $p$-subgroups of $G$ and $H$ still play an important role in determining the representation type of $k\mathcal{C}$ as illustrated by condition (3) in Theorem 1.1, which will be proved in this subsection.

By the previous proposition, throughout this subsection without loss of generality we assume that \textbf{$H$ acts transitively on $\mathcal{C} (x, y)$}. Observe that the condition that $H$ has a $p$-subgroup acting nontrivially on $\mathcal{C} (x, y)$ is equivalent to one of the following conditions:
\begin{enumerate}
\item $H$ has a Sylow $p$-subgroup acting nontrivially on $\mathcal{C} (x, y)$;
\item all Sylow $p$-subgroups of $H$ act nontrivially on $\mathcal{C} (x, y)$;
\item $O^{p'} H$, the normal subgroup generated by all Sylow $p$-subgroups of $H$, is not contained in $H_0$;
\item there is some $h \in H$ with order a power of $p$ such that $h\alpha \neq \alpha$.
\end{enumerate}
Moreover, if $G$ has a $p$-subgroup acting nontrivially on $\mathcal{C} (x, y)$, so does $H$. Indeed, since $H$ acts transitively, $G_0 \lhd G_1 = G$. Consequently, $G_0$ does not contain a Sylow $p$-subgroup of $G$ since otherwise $O^{p'}G \leqslant G_0$, contradicting the given condition. Thus $p$ divides $| G_1/ G_0 = |H_1/H_0|$, so $H_0$ does not contain a Sylow $p$-subgroup of $H_1 \leqslant H$. In conclusion, $H_1$ and $H$ have a $p$-subgroup acting nontrivially on $\mathcal{C} (x, y)$.

\begin{proposition}
Suppose that $\Char(k) = p \geqslant 5$. If both $G$ and $H$ have $p$-subgroups acting nontrivially on $\mathcal{C} (x, y)$, then $k\mathcal{C}$ is of infinite representation type.
\end{proposition}

\begin{proof}
We have assumed that $H$ acts transitively on $\mathcal{C} (x, y)$. Consequently, $G = G_1$. Define a subcategory $\mathcal{D}$ of $\mathcal{C}$ as follows:
\begin{equation*}
\xymatrix{ \mathcal{D}: & x \ar@(ld,lu)[]|{G} \ar @<1ex>[rr] ^{H_1 \alpha G} \ar@<-1ex>[rr] ^{\ldots} & & y \ar@(rd,ru)[]|{H_1}}
\end{equation*}
By Lemma 3.3, it suffices to show the infinite representation type of $k\mathcal{D}$.

We claim that both $G$ and $H_1$ act transitively on $\mathcal{D} (x, y) = H_1 \alpha G$. Indeed, by the definition of $H_1$, $H_1 \alpha \subseteq \alpha G$, so $H_1 \alpha G = \alpha G$, i.e., $G$ acts transitively on $\mathcal{D} (x, y)$. On the other hand, since $H$ acts transitively on $\mathcal{C} (x, y)$, for every $g \in G$, there exists some $h \in H$ with $h \alpha = \alpha g$. But again by the definition of $H_1$, $h \in H_1$. Therefore, $\alpha G \subseteq H_1 \alpha$. This proves the claim.

Let $\bar{G} = G/ G_0$ and $\bar{H} = H_1 /H_0$. Since $\bar{G} \cong \bar{H}$, we identify these two groups. Moreover, $p$ divides $| \bar{G} | = | \bar{H} |$. Indeed, since $G$ has a $p$-subgroup acting nontrivially on $\mathcal{C} (x, y)$, we can find some $g \in G$ with order a power of $p$ such that $\alpha g \neq \alpha$, i.e., $g \notin G_0$. Therefore, its image $\bar{g} \in \bar{G}$ has order a power of $p$, so $p$ divides $| \bar{G} |$.

Define another finite EI category $\mathcal{E}$ as follows:
\begin{equation*}
\xymatrix{ \mathcal{E}: & x \ar@(ld,lu) []|{\bar{G}} \ar @<1ex>[rr] ^{\bar{H} \bar{\alpha} \bar{G}} \ar@<-1ex>[rr] ^{\ldots} & & y \ar@(rd,ru) []|{\bar{H}}}
\end{equation*}
where $\bar{G}$ acts regularly on $\mathcal{E}(x, y)$ from the right side and $\bar{H}$ acts regularly on it from the left side. The reader can check that $\mathcal{E}$ is a quotient category of $\mathcal{D}$ by factoring out $G_0$ and $H_0$. Consequently, it is enough to show the infinite representation type of $k\mathcal{E}$.

Since $p$ divides $| \bar{G} |$, we can take a nontrivial $p$-subgroup $P \leqslant \bar{G}$. Correspondingly, $Q = \Stab _{\bar{H}} (\bar{\alpha} P) \leqslant \bar{H}$, and $P \cong Q$. Let $\mathcal{F}$ be the following subcategory of $\mathcal{E}$:
\begin{equation*}
\xymatrix{ \mathcal{F}: & x \ar@(ld,lu) []|{P} \ar @<1ex>[rr] ^{ Q \bar{\alpha} P} \ar@<-1ex>[rr] ^{\ldots} & & y \ar@(rd,ru) []|{Q}}.
\end{equation*}
By Theorem 5.3 in \cite{Dietrich2}, $k\mathcal{F}$ is of infinite representation type. By Lemma 3.5, $k\mathcal{E}$ is of infinite representation type as well. This finishes the proof.
\end{proof}

By the remark before this proposition, if $G$ has a $p$-subgroup acting nontrivially on $\mathcal{C} (x, y)$, so does $H$ since it acts transitively. Consequently, $k\mathcal{C}$ has infinite representation type. Thus in the rest of this subsection we assume that \textbf{$O^{p'}G$ acts trivially on $\mathcal{C} (x, y)$}, i.e., $O^{p'}G \leqslant G_0$. By this assumption, we know that $p$ does not divide $|G_1 : G_0| = |H_1 : H_0|$. Therefore, for every Sylow $p$-subgroup $P \leqslant H$, $P \cap H_0 = P \cap H_1$. In particular, $p$ divides $|H_1|$ if and only if it divides $|H_0|$, and $O^{p'} H \leqslant H_0$ if and only if $O^{p'} H \leqslant H_1$.

\begin{lemma}
Suppose that $G= 1$ and $p$ divides $|H : H_0|$. If $\dim _k \End _{kH} (P_k) \geqslant 6$ where $P_k$ is the projective cover of the simple $kH$-module $k$, then $k\mathcal{C}$ is of infinite representation type.
\end{lemma}

The condition that $p$ divides $|H : H_0|$ implies that there is a Sylow $p$-subgroup $P \leqslant H$ acting nontrivially on $\mathcal{C} (x, y)$. But the converse statement is not true. Indeed, $p$ does not divide $|H: H_0|$ if and only if $H_0$ contains a Sylow $p$-subgroup of $H$, weaker than the condition $O^{p'} \leqslant H_0$, which by the above remark is equivalent to saying that a Sylow $p$-subgroup of $H$ acting trivially on $\mathcal{C} (x, y)$. We also remind the reader that $P_k$ can also be viewed as a projective $k\mathcal{C}$-module and $\End_ {k\mathcal{C}} (P_k) \cong \End _{kH} (P_k)$.

\begin{proof}
Consider the projective $k\mathcal{C}$-module $P_k \oplus k\mathcal{C} 1_x$ and $\Lambda = \End _{k\mathcal{C}} (P_k \oplus k\mathcal{C} 1_x) ^{\textnormal{op}}$. By Proposition 2.5 on page 36 in \cite{Auslander}, $\Lambda$-mod is equivalent to a subcategory of $k\mathcal{C}$-mod. Thus it suffices to show the infinite representation type of $\Lambda$. We have
\begin{equation*}
\End _{k\mathcal{C}} (P_k) \cong k[t] / (t^d), \quad \End _{k\mathcal{C}} (k\mathcal{C} 1_x) \cong 1_x k\mathcal{C} 1_x \cong k, \quad \Hom _{k\mathcal{C}} (k\mathcal{C} 1_x, P_k) =0,
\end{equation*}
where $d = \dim_k \End _{kH} (P_k)$, and
\begin{equation*}
\text{dim}_k \Hom _{k\mathcal{C}} (P_k, k\mathcal{C} 1_x) = \text{dim}_k \Hom _{kH} (P_k, kH \alpha) = \text{dim}_k \Hom _{kH} (P_k, k\uparrow _{H_0}^H).
\end{equation*}
Let $M$ be the Scott module of $k\uparrow _{H_0}^H$. Since $p$ divides $|H : H_0|$, $M \ncong k$. Therefore, $M$ has at least two composition factors isomorphic to $k$. Consequently,
\begin{equation*}
t = \text{dim}_k \Hom _{k\mathcal{C}} (P_k, k\mathcal{C} 1_x) \geqslant  \text{dim}_k \Hom _{kH} (P_k, M) \geqslant 2.
\end{equation*}
Thus $\Lambda$ is isomorphic to the path algebra of the following quiver with relations $\delta^d = 0$ and $\delta^t \beta = 0$ for some $t \geqslant 2$:
\begin{equation*}
\xymatrix{ \bullet \ar[r] ^{\beta} & \bullet \ar@(rd,ru) []|{\delta}}.
\end{equation*}
From Bongartz's list in \cite{Bautista, Bongartz} we conclude that $\Lambda$ and hence $k\mathcal{C}$ are of infinite representation type if $d \geqslant 6$.
\end{proof}

Before stating and proving the next two propositions, let us do some reduction. Define a finite EI category $\mathcal{G}$ as follows:
\begin{equation*}
\xymatrix{ \mathcal{G}: & x \ar@(ld,lu)[]|{1} \ar @<1ex>[rr] ^{H \bar{\alpha} } \ar@<-1ex>[rr] ^{\ldots} & & y \ar@(rd,ru)[]|{H}}.
\end{equation*}
where $\Stab_H (\bar{\alpha}) = H_1$. Let $F: \mathcal{C} \rightarrow \mathcal{G}$ be the functor defined by: $F(g) = 1$, $F(h) = h$, $F(h \alpha g) = h \bar{\alpha}$ for $g \in G$ and $h \in H$. The reader can check that $F$ is well defined. Therefore, $\mathcal{G}$ is a quotient category of $\mathcal{C}$. Moreover, if a Sylow $p$-subgroup $P \leqslant H$ acts nontrivially on $\mathcal{C} (x, y)$, then it acts nontrivially on $\mathcal{G} (x, y)$ as well. This conclusion comes from the fact that $O^{p'} H \leqslant H_0$ if and only if $O^{p'} H \leqslant H_1$, see the remark before this lemma and the equivalent conditions introduced at the beginning of this subsection.

\begin{proposition}
Suppose that $\Char(k) = p \geqslant 5$. If $G$ (resp., $H$) has a $p$-subgroup $P$ (resp., $Q$) such that $1 \neq P \cap G_0 \neq G_0$ (resp., $1 \neq Q \cap H_0 \neq H_0$), then $k\mathcal{C}$ is of infinite representation type.
\end{proposition}

\begin{proof}
By the reduction before the previous lemma, we know $O^{p'}G \leqslant G_0$. Therefore, we only consider the case that a $p$-subgroup $Q \leqslant H$ such that $Q \cap H_0 \neq 1$ and $Q \nleqslant H_0$, i.e., $Q\alpha \neq \{ \alpha \}$ and $\Stab _Q (\alpha) \neq 1$. Moreover, since $p$ does not divide $|H_1/H_0| = |G_1/G_0|$, we conclude that $1 \neq Q \cap H_0 = Q \cap H_1$ and $Q \nleqslant H_1$. Consequently, the quotient category $\mathcal{G}$ defined above satisfies the given condition as well since $\Stab _{H} (\bar{\alpha}) = H_1$. Therefore, it suffices to show the infinite representation type of $k\mathcal{G}$. Observe that $p$ divides both $|H_1|$ and $|H : H_1|$, so the Scott module $N$ of $k \uparrow _{H_1}^H$ is neither isomorphic to $k$ nor projective. We have two cases:\\

\textbf{Case I:} $\Soc (N)$ contains a simple summand $S \ncong k$. Since $\Soc(N) \cong \Top(N)$ (see \cite{Takahashi}), $\Top (N)$ contains a summand isomorphic to $S$. Thus
\begin{equation*}
0 \neq \Hom _{kH} (N, S) \subseteq \Hom _{kH} (k \uparrow _{H_1}^H, S) \cong \Hom _{kH_1} (k, S \downarrow _{H_1}^H),
\end{equation*}
we conclude that $\Soc (S \downarrow _{H_1}^H)$ contains a summand $k_1$ isomorphic to $k$.

Write $N = k \oplus S \oplus N' = k \oplus (k_1 \oplus S') \oplus N'$ as vector spaces, where both $k$ and $S$ are simple summands in $\Soc(N)$. Note that $N$ is indecomposable, $k \ncong S$, and they both are contained in $\Soc (N)$. Therefore, every non-invertible endomorphism on $N$ maps $k \oplus S$ to 0.

Construct a class of representations $R_d$ of $\mathcal{G}$ in the following way, $0 \neq d \in k$. First, $R_d (x) = k$, $R_d(y) = N$. The map $R_d(\alpha)$ is defined by
\begin{equation*}
\xymatrix {k \ar[rr] ^-{\begin{bmatrix} 1 \\ d \\ 0 \\ 0 \end{bmatrix}} & & k \oplus (k_1 \oplus S') \oplus M'}.
\end{equation*}
By Proposition 3.1, this determines a representation of $\mathcal{G}$. It is indecomposable. Moreover, we check that $R_b \cong R_d$ if and only if $d = b$. In this way we construct infinitely many pairwise non-isomorphic indecomposable representations of $\mathcal{G}$.\\

\textbf{Case II:} $\Soc(N) \cong \Top (N) \cong k$. In this situation $N$ is a proper quotient module of the projective $kH$-module $P_k$. In particular, the multiplicity $[P_k, k] > [N: k] \geqslant 2$. Therefore, in the Brauer graph $\Gamma$ of the principal block $B_0 (kH)$, there is an exceptional vertex to which the edge $k$ is adjacent.

Let $m$ be the multiplicity of this exceptional vertex and $e$ be the number of edges in $\Gamma$. Then $e \mid (p-1)$ and $em = |D| -1$, where the defect group $D$ of $B_0 (kH)$ is a Sylow $p$-subgroup of $kH$ (see \cite{Alperin, Benson}). By the given condition, $|D| \geqslant p^2$. Therefore, $m \geqslant p + 1$, so $P_k$ has at least $p+2$ composition factors isomorphic to $k$. Consequently, $\dim_k \End _{kH} (P_k) \geqslant p+2 \geqslant 7$. The conclusion follows from Lemma 3.11 (note that $\Stab _{H} (\bar{\alpha})$ is $H_1$ instead of $H_0$ for $\mathcal{G}$).
\end{proof}

Actually, we have shown the following conclusion in the above proof:

\begin{corollary}
Suppose that $\Char(k) = p \neq 2, 3$. If $k\mathcal{C}$ is of finite representation type, then the Scott module of $k \uparrow _{H_0}^H$ (or $k \uparrow _{H_1}^H$) is either $k$ or projective.
\end{corollary}

\begin{proof}
This is clearly true for $p = 0$. Thus we assume $p \geqslant 5$. Since $p$ does not divide $|H_1 : H_0|$ by the assumption we made before, the Scott module of $k \uparrow _{H_0}^H$ is isomorphic to that of $k \uparrow _{H_1}^H$, see \cite{Broue, Green, Takahashi}. By the above proposition, $H_0$ either contains a Sylow $p$-subgroup of $H$ or only has a trivial $p$-subgroup. In the first case the Scott module is isomorphic to $k$, and in the second case it is isomorphic to the the uniserial projective module $P_k$.
\end{proof}

\subsection{Permutation modules}

In this subsection we study the structure of permutation modules to develop a few more criteria. By previous results, we can assume that $H$ \textbf{acts transitively on} $\mathcal{C} (x, y)$. Therefore, $G_1 = G$. We identify $k(G_1/G_0)$ and $k(H_1/H_0)$. \textbf{We do not suppose that $O^{p'} G$ acts trivially on this biset}, so $|H_1 : H_0| = |G_1 : G_0|$ might not be invertible in $k$, and hence $k \uparrow _{H_0}^{H_1} \cong k (H_1 /H_0)$ might not be semisimple. We will find in many situations the representation type of $k\mathcal{C}$ is determined by the structure of permutation module $k \uparrow _{H_1}^H$.

The following lemma will be used frequently.

\begin{lemma}
Let $M$ be a $kH$-module. Then $\Hom _{kH} (M, k \uparrow _{H_1}^H) \neq 0$ if and only if $k$ is a summand of $\Top (M \downarrow _{H_1}^H)$. Dually, $\Hom _{kH} (k \uparrow _{H_1}^H, M) \neq 0$ if and only if $k$ is a summand of $\Soc (M \downarrow _{H_1}^H)$.
\end{lemma}

\begin{proof}
Use Frobenius reciprocity.
\end{proof}

The following proposition generalizes Corollary 6.5 in \cite{Li1}, where $kH$ is supposed to be semisimple.

\begin{proposition}
If $k\mathcal{C}$ is of finite representation type, then the following conditions must be true:
\begin{enumerate}
\item $\Top (k \uparrow _{H_1}^H)$ has no repeated summands;
\item every indecomposable summand $M$ of $k \uparrow _{H_1}^H$ is uniserial or biserial, and if $M$ is biserial, it is projective;
\item if $M \ncong N$ are indecomposable summands of $k \uparrow _{H_1}^H$, then $\Hom _{kH} (M, N) = 0$.
\end{enumerate}
\end{proposition}

\begin{proof}
Suppose that a simple $kH$-module $U$ appears at least twice in $\Top (k \uparrow _{H_1}^H)$. Therefore, by Frobenius reciprocity $k$ appears at least twice in $\Soc (U \downarrow _{H_1}^H)$. Take two summands $k_1 \cong k \cong k_2$ in $\Top (U \downarrow _{H_1}^H)$ and write $U = k_1 \oplus k_2 \oplus U'$ as vector spaces. Define a family of representations $R_d$ for $\mathcal{C}$ as follows, $0 \neq d \in k$:
\begin{equation*}
\xymatrix{ \varphi_d: k \ar[rr] ^-{\begin{bmatrix} 1 \\ d \\ 0 \end{bmatrix}} & & k_1 \oplus k_2 \oplus U'}.
\end{equation*}
As we did in the proof of Proposition 3.12, we can check that $R_d$ is indecomposable, and that $R_d \cong R_b$ if and only if $b = d$. Thus (1) must be true.

We show that every indecomposable summand $M$ of $k \uparrow _{H_1}^H$ has a simple top. This implies (2). Indeed, since $M$ has a simple top, it is a quotient module of an indecomposable projective $kH$-module, and hence must be uniserial or biserial (see \cite{Alperin}). Because $k \uparrow _{H_1}^H$ is isomorphic to its dual module, $M$ must has a simple socle. Therefore, if $M$ is biserial, it must be isomorphic to its projective cover.

Suppose that $\Top (M)$ contains more than one simple summands. Take $U$ and $V$ from $\Top (M)$. By (1), $U$ is not isomorphic to $V$. By the previous lemma and (1), $U \downarrow _{H_1}^H$ contains exactly a simple summand $k_1 \cong k$. Similarly, $V \downarrow _{H_1}^H$ contains exactly a simple summand $k_2 \cong k$ as well.

Let $\Omega M$ be the first syzygy of $M$. Then $\Omega M \neq 0$ since $M$ is indecomposable and its top is not simple. Moreover, $U$ and $V$ appears as summands of $\Soc (\Omega M)$. Therefore, as we did in the proof of (1), using $\Omega M$ we can construct a family of infinitely many non-isomorphic indecomposable representations of $\mathcal{C}$. This contradiction shows that $M$ must have a simple top, and (2) is proved.

Now we turn to (3). Suppose that $\Hom_{kH} (M, N) \neq 0$. By (2) we can suppose that $M$ (resp., $N$) is a quotient module of an indecomposable projective module $P_U$ (resp., $P_V$). Therefore, if $\Hom _{kH} (M, N) \neq 0$, the multiplicity $[N: U] \neq 0$, so $[P_V, U] \neq 0$. This happens if and only if $U$ and $V$ lie in the same block $B$ of $kH$, and $U$ and $V$ are connected to the same vertex in the Brauer graph $\Gamma$ of $B$. Therefore, $\Gamma$ has a piece as follows:
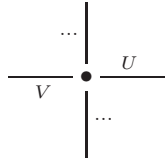
\begin{equation*}
\xymatrix { & & \\ & \bullet \ar@{-} [u]^{\ldots} \ar@{-} [r]^U \ar@{-} [l]^V \ar@{-} [d]^{\ldots} & & \\ & & }
\end{equation*}
So we can define a uniserial $kH$-module $L$ such that $\Top (L) \cong \Top (N) \cong V$, $\Soc (L) \cong \Soc (M) \cong U$, and $[L : U] = [L : V] = 1$. Therefore, $L$ is a quotient module of $P_V$, and it is actually a quotient module of $N$.

Clearly, $\Hom_{kH} (N, L) \neq 0 \neq \Hom_{kH} (M, L)$. Therefore,
\begin{equation*}
\dim_k \Hom_{kH_1} (S, L \downarrow _{H_1}^H) = \dim_k \Hom_{kH} (S \uparrow _{H_1}^H, L) \geqslant 2.
\end{equation*}
So we can find two simple summands $S_1 \cong S \cong S_2$ in $\Soc (L \downarrow _{H_1}^H)$. Note that $\End _{kH} (N) \cong k$. Using this module $L$, as we did in the proof of (1) we can construct infinitely many indecomposable non-isomorphic representations of $\mathcal{C}$. The conclusion of (3) follows from this contradiction.
\end{proof}

In the following lemma we give some algebras of infinite representation types. They will be used in the proof of the next proposition.

\begin{lemma}
The path algebras of the following quivers with relations are of infinite representation type:
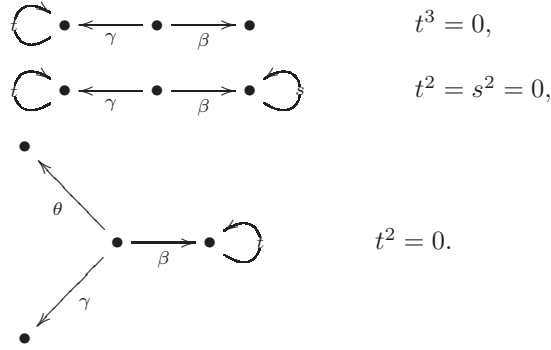
\begin{align*}
& \xymatrix{ \bullet \ar@(ld,lu)[]|{t} & \bullet \ar[l]^{\gamma} \ar[r]_{\beta} & \bullet & & t^3 = 0,}\\
& \xymatrix{ \bullet \ar@(ld,lu)[]|{t} & \bullet \ar[l]^{\gamma} \ar[r]_{\beta} & \bullet \ar@(rd,ru)[]|{s} & & t^2 = s^2 = 0,}\\
& \xymatrix{ \bullet & & \\ & \bullet \ar[ul]^{\theta} \ar[r]_{\beta} \ar[ld] ^{\gamma} & \bullet \ar@(rd,ru)[]|{t} & & t^2 = 0. \\ \bullet & & }
\end{align*}
\end{lemma}

\begin{proof}
These results can be proved by the covering theory. As example, we show the first algebra is of infinite representation type, and leave the other two to the reader. A universal covering of the first quiver is described as follows with relation $t^3 = 0$:
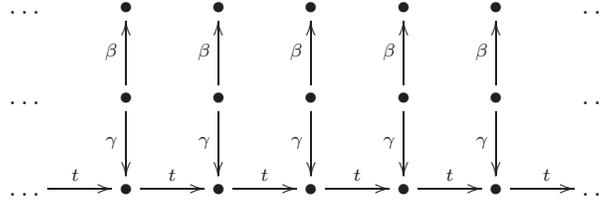
\begin{equation*}
\xymatrix{
\ldots & \bullet & \bullet & \bullet & \bullet & \bullet & \ldots\\
\ldots & \bullet \ar[u]^{\beta} \ar[d]_{\gamma} & \bullet \ar[u]^{\beta} \ar[d]_{\gamma} & \bullet \ar[u]^{\beta} \ar[d]_{\gamma} & \bullet \ar[u]^{\beta} \ar[d]_{\gamma} & \bullet \ar[u]^{\beta} \ar[d]_{\gamma} & \ldots\\
\ldots \ar[r]^t & \bullet \ar[r]^t & \bullet \ar[r]^t & \bullet \ar[r]^t & \bullet \ar[r]^t & \bullet \ar[r]^t & \ldots
}
\end{equation*}
This covering is not locally representation-finite since it contains a subquiver of infinite representation type as below:
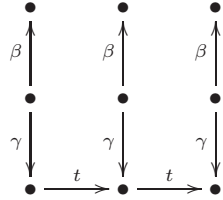
\begin{equation*}
\xymatrix{
\bullet & \bullet & \bullet \\
\bullet \ar[u]^{\beta} \ar[d]_{\gamma} & \bullet \ar[u]^{\beta} \ar[d]_{\gamma} & \bullet \ar[u]^{\beta} \ar[d]_{\gamma} \\
\bullet \ar[r]^t & \bullet \ar[r]^t & \bullet
}
\end{equation*}
Therefore, the first algebra is of infinite representation type.
\end{proof}

\begin{proposition}
If $k\mathcal{C}$ is of finite representation type, then $\dim_k \End _{kH} (k \uparrow _{H_1}^H) \leqslant 3$. That is, $H_1 \backslash H / H_1$ has at most three double cosets.
\end{proposition}

The second statement comes from
\begin{align*}
\Hom _{kH} (k \uparrow _{H_1}^H, k \uparrow _{H_1}^H) & \cong \Hom _{kH_1} (k \uparrow _{H_1}^H \downarrow _{H_1}^H, k)\\
& \cong \bigoplus _{s \in H_1 \backslash H / H_1} \Hom _{kH_1} ((s \otimes k) \uparrow _{H_1 \cap s H_1 s^{-1}} ^{H_1}, k)\\
& \cong \bigoplus _{s \in H_1 \backslash H / H_1} \Hom _{k(H_1 \cap s H_1 s^{-1})} (s \otimes k, k).
\end{align*}
So $\dim_k \End _{kH} (k \uparrow _{H_1}^H)$ coincides with the number of double cosets of $H_1 \backslash H / H_1$.

\begin{proof}
Note that $\mathcal{C}$ has a quotient category $\mathcal{G}$ (defined before Proposition 3.12) such that $\mathcal{G} (x, x) = 1$, $\mathcal{G} (y, y) = H$, and $\mathcal{C} (x, y) = H \bar{\alpha}$ where $\Stab_H (\bar{\alpha}) = H_1$. Therefore, it suffices to show that $k\mathcal{G}$ is of infinite representation type if $\dim_k \End _{kH} (k \uparrow _{H_1}^H) \geqslant 4$.

Denote all indecomposable summands of $k \uparrow _{H_1}^H$ by $M_1, M_2, \ldots, M_n$. By Proposition 3.15, each of them is a quotient module of an indecomposable projective module of $kH$. Let $e_1, e_2, \ldots, e_n$ be primitive idempotent elements in $kH$ such that $P_i = kHe_i$ is a projective cover of $M_i$, $1 \leqslant i \leqslant n$. By Proposition 6.4, we can assume the following properties: $n \leqslant 3$; $\Hom_{kH} (P_i, M_j) = 0$, $1 \leqslant i \neq j \leqslant n$. Consider the algebra
\begin{equation*}
\Lambda = \End _{k\mathcal{G}} (k\mathcal{G} (1_x + e_1 \ldots + e_n)) ^{\textnormal{op}} \cong (e_1 + \ldots + e_n + 1_x) k\mathcal{G} (e_1 + \ldots + e_n + 1_x).
\end{equation*}
We only need to show that $\Lambda$ is of infinite representation type if $\dim_k \End _{kH} (k \uparrow _{H_1}^H) \geqslant 4$ since by Proposition 2.5 on page 36 in \cite{Auslander} $\Lambda$-mod is equivalent to a subcategory of $k\mathcal{G}$-mod.

Observe that $1_x k\mathcal{G} (e_1 + \ldots e_n) = 0$, $1_x k\mathcal{G} 1_x \cong k$, and
\begin{align*}
(e_1 + \ldots + e_n) k\mathcal{G} 1_x & = (e_1 + \ldots + e_n) kH \bar{\alpha}\\
& \cong \Hom_{kH} (kH(e_1 + \ldots + e_n), kH \bar{\alpha})\\
& \cong \Hom_{kH} (kH(e_1 + \ldots + e_n), k \uparrow _{H_1}^H)\\
& \cong \Hom_{kH} (kH(e_1 + \ldots + e_n), M_1 \oplus \ldots \oplus M_n)\\
& \cong \Hom_{kH} (P_1 + \ldots + P_n, M_1 \oplus \ldots \oplus M_n)\\
& \cong \bigoplus _{i=1}^n \Hom_{kH} (P_i, M_i)
\end{align*}
as left $(e_1 + \ldots + e_n) k\mathcal{G} (e_1 + \ldots + e_n)$-modules since $\Hom _{k\mathcal{G}} (P_i, M_j) = 0$ for $1 \leqslant i \neq j \leqslant n$. By the same reasoning, the product of an element in the two-sided $\Lambda$-ideal generated by $\oplus _{1 \leqslant i \neq j \leqslant n} \Hom_{kH} (P_i, P_j)$ and an element in $(e_1 + \ldots + e_n) k\mathcal{G} 1_x$ is 0. We also have:
\begin{equation*}
\dim_k \Hom_{kH} (kHe_i, M_i) = \dim_k \End_{kH} (M_i).
\end{equation*}
Therefore, modulo the two-sided $\Lambda$-ideal generated by $\oplus _{1 \leqslant i \neq j \leqslant n} \Hom_{kH} (P_i, P_j)$, $\Lambda$ has a quotient algebra $\bar{\Lambda}$ which is isomorphic to the path algebra of the following bounded quiver
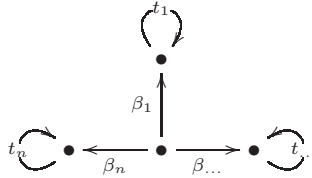
\begin{equation*}
\xymatrix{ & \bullet \ar@(ul,ur)[]|{t_1} \\ \bullet \ar@(ld,lu)[]|{t_n} & \bullet \ar[u]^{\beta_1} \ar[r]_{\beta_{\ldots}} \ar[l] ^{\beta_n} & \bullet \ar@(rd,ru)[]|{t_{\ldots}} }
\end{equation*}
with relations $t_i^{a_i} = 0 = t_i^{d_i} \beta_i$, where $a_i \geqslant d_i = \dim_k \End _{kH} (M_i)$, $1 \leqslant i \leqslant n$.

Suppose that $d = \dim_k \End _{kH} (k \uparrow _{H_1}^H) \geqslant 4$. Note that $d = \sum _{i=1}^n d_i$. We have several cases:

\textbf{Case I:} $n = 1$. Therefore, $k \uparrow _{H_1}^H $ is an indecomposable $kH$-module, and $\bar{\Lambda}$ is isomorphic to the path algebra of the following quiver with relation $t^a = 0 = t^d \beta$, $a \geqslant d$. we deduce that $\bar{\Lambda}$ is of infinite representation type from Bongartz's list in \cite{Bautista, Bongartz} if $d \geqslant 4$.
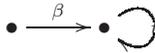
\begin{equation*}
\xymatrix{ \bullet \ar[r]^{\beta} & \bullet \ar@(ru,rd)[]|{t}}
\end{equation*}

\textbf{Case II:} $n = 2$, so $d_1 + d_2 = d \geqslant 4$. If one of them is 1, then $\bar{\Lambda}$ has the first algebra in the previous lemma as a quotient algebra; if both numbers are at least 2, then $\bar{\Lambda}$ has the second algebra in the previous lemma as a quotient algebra.

\textbf{Case III:} $n = 3$, so $d_1 + d_2 + d_3 \geqslant 4$. Therefore, at least one of them must be bigger than 1, and $\bar{\Lambda}$ has the third algebra in the previous lemma as a quotient algebra.

Consequently, in all cases we get infinite representation type for $\bar{\Lambda}$. This contradiction tells us that $d \leqslant 3$.
\end{proof}

\section{Classifications of representation types.}

As in the previous section $\mathcal{C}$ is a connected skeletal finite EI category with two objects $x$ and $y$ such that $\mathcal{C} (y, x) = \emptyset$. Let $\alpha$, $G_0 \lhd G_1 \leqslant G$ and $H_0 \lhd H_1 \leqslant H$ as defined before. In this section we use the criteria developed in the previous section to classify the representation type of $\mathcal{C}$. We point out that this classification is not complete. However, it covers most cases.

When both $|G|$ and $|H|$ are invertible, we can use the algorithm described in \cite{Li1} to construct the ordinary quiver $Q$ of $k\mathcal{C}$. Moreover, we proved in that paper that $k\mathcal{C}$ is Morita equivalent to $kQ$. Therefore, the representation type of $k\mathcal{C}$ can be determined by Gabriel's theorem.

When both $G$ and $H$ are $p$-groups, this classification has been described in \cite{Dietrich1}. We record his result here.

\begin{proposition}
(Theorem 5.3 in \cite{Dietrich2}) Let $\mathcal{C}$ be a finite EI category with two objects for which the automorphism groups are $p$-groups. Suppose that $\mathcal{C} (y, x) = \emptyset$. Let $G = \mathcal{C} (x, x)$ and $H = \mathcal{C} (y, y)$. Then $k\mathcal{C}$ is of finite representation type if and only if the following conditions are satisfied
\begin{enumerate}
\item both $G$ and $H$ are cyclic;
\item either $G$ or $H$ acts transitively on $\mathcal{C} (x, y)$;
\item one of the following is true:
\begin{enumerate}[(a)]
\item $| \mathcal{C} (x, y) | \leqslant 1$;
\item $|G| |H| \leqslant 3$;
\item $p = 2 = \mathcal{C} (x, y)$, either $G$ or $H$ is trivial;
\item $p = 2 = \mathcal{C} (x, y)$, either $G$ or $H$ has order 2 and acts transitively;
\item $p = 3 = | \mathcal{C} (x, y) | = |G| = |H|$, and both $G$ and $H$ act transitively.
\end{enumerate}
\end{enumerate}
\end{proposition}

An immediate corollary is:

\begin{corollary}
Suppose that $p \geqslant 5$. Let $\mathcal{C}$ be a skeletal finite EI category for which the automorphism groups of all objects are $p$-groups. If $k\mathcal{C}$ is of finite representation type, then $| \mathcal{C} (x, y) | \leqslant 1$ for every pair of distinct objects $x$ and $y$.
\end{corollary}

\begin{proof}
Consider the full subcategory with objects $x$ and $y$ and use the previous proposition.
\end{proof}

Now suppose that both $G$ and $H$ are arbitrary. By the previous section, in many cases the representation type of $k\mathcal{C}$ is determined by two pieces of information: the transitivity of the actions by $G$ and $H$ on $\mathcal{C} (x, y)$, and the triviality of the actions by Sylow $p$-subgroups in $G$ and $H$. Indeed, the combination of conditions (a-c)
\begin{enumerate}[ \indent (a)]
\item Both $G$ and $H$ act transitively.
\item One of $G$ and $H$ acts transitively, and the other one does not.
\item Neither $G$ nor $H$ acts transitively.
\end{enumerate}
and conditions (1-3)
\begin{enumerate}
\item Both $O^{p'}G$ and $O^{p'}H$ act trivially.
\item One of $O^{p'}G$ and $O^{p'}H$ acts trivially, and the other one does not.
\item Neither $O^{p'}G$ nor $O^{p'}H$ acts trivially.
\end{enumerate}
give us 8 situations (the combination (a)+(2) cannot happen). By Theorem 1.1, if $\Char(k) = p \neq 2, 3$, we get infinite representation type for five cases: (c)+(1), (c)+(2), (c)+(3), (a)+(3), (b)+(3). In the first subsection we will prove the finite representation type for the case (a)+(1).

\subsection{Type (a) + (1).}

We first prove the following lemma.

\begin{lemma}
If $\mathcal{C} (x, y)$ has only one morphism, then $k\mathcal{C}$ is of finite representation type.
\end{lemma}

\begin{proof}
Let $p = \Char (k)$. Take a Sylow $p$-subgroup $P$ of $G$ and a Sylow $p$-subgroup $Q$ of $H$. This can always be done since if $G$ or $H$ has order invertible in $k$, we then take the trivial group as the unique Sylow $p$-subgroup by our convention. Let $\mathcal{S}$ be the following subcategory:
\begin{equation*}
\xymatrix{ \mathcal{S}: & x \ar@(ld,lu)[]|{P} \ar[rr] ^{\alpha} & & y \ar@(rd,ru)[]|{Q}},
\end{equation*}
and define another subcategory $\mathcal{J}$ as below:
\begin{equation*}
\xymatrix{ \mathcal{J}: & x \ar@(ld,lu)[]|{G} \ar[rr] ^{\alpha} & & y \ar@(rd,ru)[]|{Q}}.
\end{equation*}

By Proposition 4.1, $k\mathcal{S}$ is of finite representation type. Applying Proposition 3.8 to the pair $(\mathcal{J}, \mathcal{S})$, we conclude that $k\mathcal{J}$ is of finite representation type, and so is the category algebra of the opposite category $\mathcal{J} ^{\textnormal{op}}$. Applying Proposition 3.8 again to the pair $(\mathcal{C} ^{\textnormal{op}}, \mathcal{J} ^{\textnormal{op}})$, we obtain the finite representation type of $k\mathcal{C} ^{\textnormal{op}}$. Therefore, $k\mathcal{C}$ is of finite representation type.
\end{proof}

Now we prove Theorem 1.2.

\begin{theorem}
Suppose that both $G$ and $H$ act transitively on $\mathcal{C} (x, y)$. If $\Char(k) = p \neq 2, 3$, then $k\mathcal{C}$ is of finite representation type if and only if all $p$-subgroups of $G$ and $H$ act trivially on $\mathcal{C} (x, y)$.
\end{theorem}

\begin{proof}
Since both $G$ and $H$ act transitively on $\mathcal{C} (x, y)$, we have $G_0 \lhd G_1 = G$, $H_0 \lhd H_1 = H$. If $G$ or $H$ has a $p$-subgroup acting nontrivially on $\mathcal{C} (x, y)$, so does the other one by the remark before Proposition 3.10. Therefore, $k\mathcal{C}$ has infinite representation type by Proposition 3.10.

On the other hand, if all $p$-groups of $G$ and $H$ act trivially on $\mathcal{C} (x, y)$, then $| G/G_0 | = | H / H_0 | = n$ is invertible in $k$. Let $\mathcal{K}$ be the following subcategory:
\begin{equation*}
\xymatrix{ \mathcal{K}: & x \ar@(ld,lu)[]|{G_0} \ar[rr] ^{\alpha} & & y \ar@(rd,ru) []|{H_0}},
\end{equation*}
whose category algebra is of finite representation type by the previous lemma. Applying Proposition 3.8 to the pair $(\mathcal{C}, \mathcal{K})$, we conclude that $k\mathcal{C}$ is of finite representation type as well.
\end{proof}

\subsection{Both $G$ and $H$ are abelian.}

Throughout this subsection we assume that $G$ and $H$ are abelian, and $\Char(k) = p \neq 2, 3$. As before, we suppose that $H$ acts transitively on $\mathcal{C} (x, y)$. Note that when $p = 0$, by our convention $1$ is the unique $p$-subgroup of a finite group.

We first describe some easy observation.

\begin{lemma}
Suppose that $k\mathcal{C}$ is of finite representation type and $\Char(k) = p \neq 2, 3$. Then both $O^{p'}G$ and $O^{p'}H$ act trivially on $\mathcal{C} (x, y)$, and $|H : H_1| \leqslant 3$.
\end{lemma}

\begin{proof}
Since $\Char(p) = p \neq 2, 3$, we know that $O^{p'} G$ must act trivially on $\mathcal{C} (x, y)$. Let $P$ be the unique Sylow $p$-subgroup of $H$. Without loss of generality we assume that $P \neq 1$ (so $p \geqslant 5$). By Proposition 3.12, either $P \leqslant H_0 \leqslant H_1$ or $P \cap H_0 = 1$. We show the second possibility cannot happen, so $P \leqslant H_0$ acts trivially on $\mathcal{C} (x, y)$, as asserted by the first statement

Otherwise, if $P \cap H_0 = 1$, then $P \cap H_1 = 1$ as well (see the remark before Lemma 3.11). Write $H = P \times H'$. By our assumption, we can choose $H_1 \leqslant H'$. Factoring out $H'$, we get a quotient category $\mathcal{Q}$ with $\mathcal{Q} (x, x) = 1$, $\mathcal{Q} (y, y) = P$, and $\mathcal{Q} (y, y)$ acts regularly on $\mathcal{Q} (x, y)$. Thus $k\mathcal{Q}$ is of infinite representation type, so does $k\mathcal{C}$, contradicting the given condition.

The second statement comes from Proposition 3.17.
\end{proof}

Therefore, if $k\mathcal{C}$ is of finite representation type, $H / H_1 \cong C_n$ where $C_n$ is a cyclic group of order $n$, $1 \leqslant n \leqslant 3$.

From now on we suppose that $O^{p'}G \leqslant G_0$ and $O^{p'}H \leqslant H_0$. Note that $k(G / G_0) = k(G_1/G_0)$ is a semisimple $kG$-module with pairwise non-isomorphic simple summands. Let $S_1, \ldots, S_r$ be all simple summands. Correspondingly, let $e_1, \ldots, e_r$ be primitive idempotents in $kG$ such that $kGe_i$ is a projective cover of $S_i$, $1 \leqslant i \leqslant r$. Denote $1 - e_1 - \ldots - e_r$ by $f$. These simple modules $S_i$ can be viewed as $kH_1$-modules as well since $G/G_0 \cong H_1/H_0$, and we let $E_i$ be idempotents in $kH$ such that $kHE_i$ is the projective cover of $kH \otimes_{kH_1} S_i$, $1 \leqslant i \leqslant r$. By the previous lemma, $H / H_1 \cong C_n$ with $n \leqslant 3$, so $kH \otimes _{kH_1} S_i$ has $n$ simple summands, and each $E_i$ is a sum of $n$ primitive idempotents in $kH$. Let $F = 1 - E_1 - \ldots - E_r$. Note that
\begin{equation*}
kH \alpha \cong k \uparrow _{H_0}^H \cong kH \otimes _{kH_1} (k \uparrow _{H_0}^{H_1}) \cong \oplus _{i=1}^r kH \otimes _{kH_1} S_i.
\end{equation*}

\begin{lemma}
Notation as above. Then $k\mathcal{C} = \oplus _{i = 1}^ r \Lambda_i \oplus kHF \oplus kGf$ as algebras, where $\Lambda_i = kHE_i \oplus kGe_i \oplus kHE_i\alpha$ as vector spaces.
\end{lemma}

\begin{proof}
Since $\mathcal{C} (x, y) = H \alpha$, $1_y = F + E_1 + \ldots + E_r$, and
\begin{align*}
KHF\alpha & \cong \oplus _{i=1}^r F(kH \otimes _{kH_1} S_i) = 0
\end{align*}
since $kH \otimes _{kH_1} S_i$ has a projective cover $kHE_i$. Therefore, the first decomposition is true as vector spaces, and we want to show that all summands are $k\mathcal{C}$-ideals. Particularly, it suffices to show the following identities: $kHE_i \alpha G \subseteq kHE_i \alpha$ and $kH\alpha G e_i \subseteq kHE_i \alpha$ for $1 \leqslant i \leqslant r$; $kHF \alpha G = 0$; and $kH\alpha Gf=0$.

First, we have
\begin{align*}
kH \alpha G f & = kH \alpha f = \oplus _{i=1}^r kH \otimes _{kH_1} S_if = 0
\end{align*}
since $f = 1 - \sum_{i=1}^ r e_i$ and $kGe_i$ is the projective cover of $S_i$.

The identity $KHF \alpha G = 0$ is clear as we have shown $kHF\alpha = 0$. The identity $kH E_i \alpha G \subseteq kHE_i \alpha$ is also clear since $\alpha G \subseteq H\alpha$ and $H$ is abelian. We claim $kH \alpha e_i  = kHE_i \alpha = kHE_i \alpha e_i$, implying $kH \alpha G e_i = kH \alpha e_i \subseteq kHE_i \alpha$. Indeed,
\begin{align*}
E_i kH \alpha (1-e_i) & \cong \oplus _{j=1}^r E_i (kH \otimes _{kH_1} S_j) (1-e_i)\\
& \cong E_i (kH \otimes _{kH_1} (S_1 \oplus \ldots \oplus S_{i-1} \oplus S_{i+1} \oplus \dots \oplus S_n))\\
& \cong E_i (S_1 \uparrow _{H_1} ^H \oplus \ldots \oplus S_{i-1} \uparrow _{H_1} ^H \oplus S_{i+1} \uparrow _{H_1} ^H \oplus \ldots \oplus S_n \uparrow _{H_1} ^H)
\end{align*}
is 0 since $kH E_i$ is the projective cover of $S_i \uparrow _{H_1}^H$, so $kHE_i \alpha = kHE_i \alpha e_i$. Similarly, we can show $(1- E_i) kH\alpha e_i = 0$. Therefore, $kH \alpha e_i = kHE_i \alpha e_i$.
\end{proof}

We see at once that $k\mathcal{C}$ is of finite representation type if and only if all $\Lambda_i$ are of finite representation type, $1 \leqslant i \leqslant r$, for which the structure is described by the following lemma.

\begin{lemma}
Let $\Lambda_i$ be as in the previous lemma, $1 \leqslant i \leqslant r$. Then it is isomorphic to the path algebra of the following bounded quiver with relations: $\delta_i^t = \delta_i \alpha_i = \alpha_i \theta = \theta^s = 0$, $1 \leqslant i \leqslant n = |H: H_1|$, where $t = |O^{p'} H|$ and $s = |O^{p'}G|$.
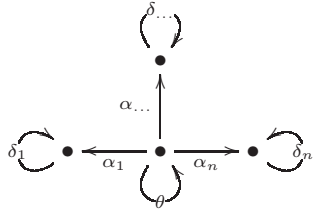
\begin{equation*}
\xymatrix{ & \bullet \ar@(ul,ur)[]| {\delta_{\ldots}} \\ \bullet \ar@(ld,lu)[]| {\delta_1} & \bullet \ar@(ld,rd)[]| {\theta} \ar[u]^{\alpha_{\ldots}} \ar[r]_{\alpha_n} \ar[l] ^{\alpha_1} & \bullet \ar@(rd,ru)[]| {\delta_n} }
\end{equation*}
\end{lemma}

\begin{proof}
Note that $H / H_1 \cong C_n$. Therefore, $S_i \uparrow _{H_1}^H$ is semisimple and has $n$ non-isomorphic simple summands, and $kHE_i = E_i kH E_i$ is a direct sum of endomorphism algebras of $n$ non-isomorphism projective $kH$-modules. Consequently, $kH E_i \cong \oplus _{i=1}^n k[\delta_i] / (\delta_i^t)$ where $t = | O^{p'}H |$. Similarly, $kG e_i \cong e_i kG e_i$ is the endomorphism algebra of an indecomposable projective $kG$-module. Therefore, $kGe_i \cong k[\theta] / (\theta^s)$ where $s = | O^{p'}G |$.

As a left $E_ikHE_i$-module, we have
\begin{align*}
E_i kH \alpha & \cong \Hom _{kH} (kHE_i, kH\alpha) \cong \Hom _{kH} (kHE_i, \oplus _{j=1}^r (S_j \uparrow _{H_1}^H))\\
& \cong \Hom _{kH} (kHE_i, S_i \uparrow _{H_1}^H)
\end{align*}
which is the space of all homomorphisms from $kHE_i$ to its top $S_i \uparrow _{H_1}^H$. Since $kHE_i \alpha  = kHE_i \alpha e_i$ from the proof of the previous lemma, as an $e_ikGe_i$-module,
\begin{align*}
E_i kH \alpha e_i & \cong \oplus _{j=1}^r E_i (kH \otimes _{kH_1} S_j) e_i \cong = (kH \otimes _{kH_1} S_i ) e_i \\
& \cong (S_i e_i)^d \cong \Hom_{e_ikGe_i} (e_ikG, S_i)^d.
\end{align*}
which is the direct sum of $d$ copies of projections from $e_ikG$ to its top $S_i$, where $d = |H: H_1|$. The conclusion follows from these observations.
\end{proof}

Now we are ready to classify the representation types of finite EI categories with two objects whose automorphism groups are abelian for $p \neq 2, 3$.

\begin{theorem}
Suppose that both $G$ and $H$ are abelian, $\Char (k) = p \neq 2,3$. Without loss of generality assume that $H$ acts transitively on $\mathcal{C} (x, y)$. Let $s = | O^{p'}G |$, $t = | O^{p'}H |$, and $n = | H: H_1 |$.
Then $k\mathcal{C}$ is of finite representation type if and only if both $O^{p'}G$ and $O^{p'}H$ act trivially on $\mathcal{C} (x, y)$, and one of the following conditions hold:
\begin{enumerate}
\item $n = 1$ for $s, t \geqslant 5$;
\item $n \leqslant 2$ for $s = 1, t \geqslant p$ or $t = 1, s \geqslant p$;
\item $n \leqslant 3$ for $t = s = 1$.
\end{enumerate}
\end{theorem}

\begin{proof}
By Lemma 4.5, we know that $n$ cannot be bigger than 3 for all cases. Lemmas 4.6 and 4.7 tell us that it suffices to consider the representation types of path algebras described in Lemma 4.8. Now we can use covering theory, Auslander-Reiten quivers, and Proposition 4.1 to get the conclusion.

As an example, let us show the case $t = 1$, $s \geqslant p \geqslant 5$, and $n = 3$ is of infinite representation type. In that case the corresponding path algebra has the following covering:
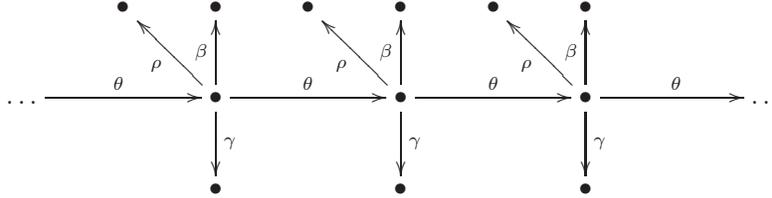
\begin{equation*}
\xymatrix{ & \bullet & \bullet & \bullet & \bullet & \bullet & \bullet \\
\ldots \ar[rr]^{\theta} & & \bullet \ar[rr]^{\theta} \ar[ul]^{\rho} \ar[u]^{\beta} \ar[d]^{\gamma} & & \bullet \ar[rr]^{\theta} \ar[ul]^{\rho} \ar[u]^{\beta} \ar[d]^{\gamma} & & \bullet \ar[rr]^{\theta} \ar[ul]^{\rho} \ar[u]^{\beta} \ar[d]^{\gamma} & & \ldots\\
 & & \bullet & & \bullet & & \bullet
}
\end{equation*}
with relations $\theta^s = \rho \theta = \beta \theta = \gamma \theta = 0$. This covering is not locally representation-finite since it has a subquiver as
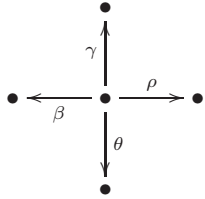
\begin{equation*}
\xymatrix{ & \bullet \\ \bullet & \bullet \ar[r]^{\rho} \ar[l]^{\beta} \ar[u]^{\gamma} \ar[d]^{\theta} & \bullet \\ & \bullet}.
\end{equation*}
Therefore, $k\mathcal{C}$ has infinite representation type for $t = 1$, $s \geqslant 5$ and $n = 3$.
\end{proof}

\section{Finite EI categories with three objects}

The reader can see from our previous analysis that the problem to classify the representation types of category algebras of arbitrary finite EI categories is very complicated. Indeed, the answer for finite EI categories with two objects is still not completely obtained. However, for an arbitrary finite EI category with more than two objects, by considering its full subcategories with two objects, we can still deduce a lot of useful information on the representation type of its category algebra. In this section we consider two special collections.

\subsection{Finite EI categories with three objects whose automorphism groups of objects are $p$-groups.} Again, we assume that $p \geqslant 5$. The main techniques we use are covering theory (see \cite{Bongartz}) and representations of \textit{string algebras} (see \cite{Butler}). By Proposition 4.1 and Corollary 4.2, we can make the following assumptions: all automorphism groups are cyclic, and there is at most one morphism between every pair of distinct objects.

Consider two families of finite EI categories $\mathcal{C}$ as below:
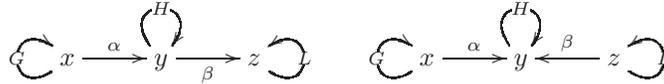
\begin{equation*}
\xymatrix{ x \ar@(ld,lu)[]|{G} \ar[r]^{\alpha} & y \ar@(lu,ru)[]|{H} \ar[r]_{\beta} & z \ar@(rd,ru)[]| {L} & & x \ar@(ld,lu)[]|{G} \ar[r]^{\alpha} & y \ar@(lu,ru)[]|{H} & z \ar@(rd,ru)[]| {L} \ar[l]_{\beta} }
\end{equation*}
Let $|G| = p^r$, $|H| = p^s$, and $|L| = p^t$. Note that the category algebra $k\mathcal{C}$ is isomorphic to the path algebra of one of the following bounded quivers with relations $g^{p^r} = h^{p^s} = l^{r^t} = 0$, $h \alpha = \alpha g =0$, and $l \beta = \beta h = 0$ (or $\beta l = h \beta = 0$ for the second class).
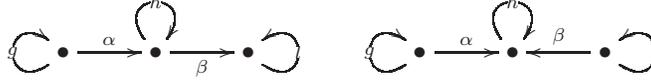
\begin{equation*}
\xymatrix{\bullet \ar@(ld,lu)[]|{g} \ar[r]^{\alpha} & \bullet \ar@(lu,ru)[]|{h} \ar[r]_{\beta} & \bullet \ar@(rd,ru)[]| {l} & & \bullet \ar@(ld,lu)[]|{g} \ar[r]^{\alpha} & \bullet \ar@(lu,ru)[]|{h} & \bullet \ar@(rd,ru)[]| {l} \ar[l]_{\beta} }
\end{equation*}

\begin{proposition}
Let $\mathcal{C}$ be a connected skeletal finite EI category with three objects such that the automorphism groups of all objects are $p$-groups and suppose that $\Char(k) = p \geqslant 5$. Then $k\mathcal{C}$ is of finite representation type if and only if $\mathcal{C}$ or $\mathcal{C}^{\textnormal{op}}$ satisfies one of the following conditions:
\begin{enumerate}
\item it lies in the first family;
\item it lies in the second family and $s = 0$;
\item it lies in the second family and $r = t = 0$.
\end{enumerate}
\end{proposition}

\begin{proof}
If $k\mathcal{C}$ is of finite representation type, then either $\mathcal{C}$ or $\mathcal{C} ^{\textnormal{op}}$ lies in one of these two families. Without loss of generality we suppose that $\mathcal{C}$ are in one of these two families. There are four cases:\\

\textbf{Case I:} If $\mathcal{C}$ is in the first family, then $k\mathcal{C}$ is a \textit{string algebra}. For a definition of string algebras, see \cite{Butler}. In that paper an algorithm to construct all indecomposable representations (up to isomorphism) of a string algebra is described. Thus we can describe all indecomposable representations $R$ of $\mathcal{C}$ as follows: for $w \in \{ x, y, z \}$, $R(w)$ is an indecomposable representation of the automorphism group of the corresponding object; the maps $R(\alpha)$ and $R(\beta)$ send the top (if nonzero) of a group representation isomorphically onto the socle (if nonzero) of another group representation. Obviously, $k\mathcal{C}$ is of finite representation type.\\

\textbf{Case II:} $\mathcal{C}$ lies in the second family, and $s = 0$. In this case $k\mathcal{C}$ is still a string algebra, and we can describe all indecomposable representations $R$ of $\mathcal{C}$ as above. Similarly, the value of $R$ on each object is an indecomposable representation of the automorphism group of this object, so $k\mathcal{C}$ is of finite representation type as well.\\

\textbf{Case III:} $\mathcal{C}$ lies in the second family, $s \neq 0$, and either $r \neq 0$ or $t \neq 0$. Without loss of generality we assume that $r \neq 0$. Then $k\mathcal{C}$ has a covering (see \cite{Bongartz}) as shown below with relations $g^{p^r} = h^{p^s} = l^{g^t} = 0$, $\alpha g = h \alpha =0$, and $h \beta = \beta l =0$. Here the arrow marked by $l$ does not exist if $t = 0$.
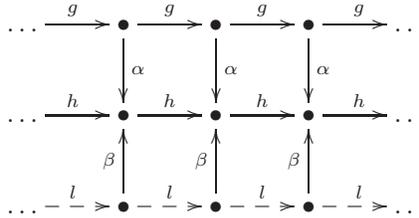
\begin{equation*}
\xymatrix{
\ldots \ar[r]^g & \bullet \ar[r]^g \ar[d]^{\alpha} & \bullet \ar[r]^g \ar[d]^{\alpha} & \bullet \ar[r]^g \ar[d]^{\alpha} & \ldots\\
\ldots \ar[r]^h & \bullet \ar[r]^h & \bullet \ar[r]^h & \bullet \ar[r]^h & \ldots\\
\ldots \ar@{-->}[r]^l & \bullet \ar@{-->}[r]^l \ar[u]^{\beta} & \bullet \ar@{-->}[r]^l \ar[u]^{\beta} & \bullet \ar@{-->}[r]^l \ar[u]^{\beta} & \ldots
}
\end{equation*}
This quiver is not locally representation-finite since it has the following subquiver with $2p+1 \geqslant 11$ vertices, which is of infinite representation type:
\begin{equation*}
\xymatrix{ 1 \ar[r]^h & 2 \ar[r]^h & \ldots \ar[r]^h & p & p+1 \ar[l]^{\alpha} \ar[r]^g & \ar[r] p+2 \ar[r]^g & \ldots \ar[r]^g & 2p\\
 & & & 0 \ar[u]^{\beta}
 }
\end{equation*}
Therefore, $k\mathcal{C}$ is of infinite representation type.\\

\textbf{Case IV:} $\mathcal{C}$ lies in the second family, and $r = t = 0$. We claim that the category algebra of this category is of finite representation type. To prove this, let $R$ be an indecomposable representation of $\mathcal{C}$. It suffices to show that $R(y)$ is an indecomposable $kH$-module. Otherwise, $R(y) = M \oplus N$ where $M, N \in kH$-mod are nonzero and $M$ is indecomposable. Note that $R(\alpha)$ and $R(\beta)$ are injective, and their images are contained in $\Soc (R(y))$. Take $0 \neq v \in \Soc(M)$, and let $u$ and $w$ be the preimages of $v$ in $R(x)$ and $R(z)$ respectively (if a preimage does not exist, we take 0 alternately). Then the reader can check that the $k\mathcal{C}$-submodule $R'$ of $R$ with $R'(x) = \langle u \rangle$, $R'(y) = M$ and $R'(z) = \langle w \rangle$ is a nonzero proper summand of $R$. This contradicts the assumption that $R$ is indecomposable. Therefore, $R(y)$ is indecomposable, so $k\mathcal{C}$ is of finite representation type.
\end{proof}

\subsection{Finite free EI categories with three objects.}

The concepts of \textit{unfactorizable morphisms} and \textit{finite free EI categories} are introduced in \cite{Li1}.

\begin{definition}
A morphism $\alpha: x \rightarrow z$ in a finite EI category $\mathcal{C}$ is unfactorizable if $\alpha$ is not an isomorphism and whenever it has a factorization as a composite $\xymatrix{x \ar[r]^{\beta} & y\ar[r]^{\gamma} & z}$, then either $\beta$ or $\gamma$ is an isomorphism.
\end{definition}

For finite EI categories, \textit{unfactorizable morphisms} are precisely \textit{irreducible morphisms} which are widely used in \cite{Auslander, Bautista, Xu1}. But in a more general context, they are different, see Example 2.4 in \cite{Li1}.

\begin{definition}
A finite EI category $\mathcal{C}$ is called a finite free EI category if it satisfies the following Unique Factorization Property (UFP): whenever a non-isomorphism $\alpha$ has two decompositions into unfactorizable morphisms,
\begin{equation*}
\xymatrix{x=x_0 \ar[r]^{\alpha_1} & x_1 \ar[r]^{\alpha_2} & \ldots \ar[r]^{\alpha_m} &x_m=y \\
x=x_0 \ar[r]^{\beta_1} & y_1 \ar[r]^{\beta_2} & \ldots \ar[r]^{\beta_n} & y_n=y}
\end{equation*}
then $m=n$, $x_i = y_i$, and there are $ h_i \in \mathcal{C} (x_i, x_i)$ such that the following diagram commutes, $1 \leqslant i \leqslant n-1$:
\begin{align*}
\xymatrix{ x_0 \ar[r]^{\alpha_1} \ar@{=}[d]^{id} & x_1 \ar[r]^{\alpha_2} \ar[d]^{h_1} & \ldots \ar[r]^{\alpha_{\ldots}} \ar[d]^{h_{\ldots}} & x_{n-1} \ar[r]^{\alpha_n} \ar[d]^{h_{n-1}} & x_n \ar@{=}[d]^{id} \\
x_0 \ar[r]^{\beta_1} & x_1 \ar[r]^{\beta_2} & \ldots \ar[r]^{\beta_{\ldots}} & x_{n-1} \ar[r]^{\beta_n} & x_n}.
\end{align*}
\end{definition}

Note that every finite free EI category $\mathcal{C}$ can be graded as follows: automorphisms lie in degree 0, and unfactorizable morphisms are in degree 1. The reader can easily check the following fact: an arbitrary finite EI category $\mathcal{C}$ is a finite free EI category if and only if $\mathcal{C}$ can be graded in the above way, and the graded category algebra $k\mathcal{C}$ is isomorphic to the tensor algebra generated by the degree 0 and the degree 1 components.

Some properties of finite free EI categories are collected below.

\begin{proposition}
Let $\mathcal{C}$ be a finite EI category.
\begin{enumerate}
\item There is a finite free EI category $\hat{\mathcal{C}}$ and a full functor $\hat{F}: \hat{\mathcal{C}} \rightarrow \mathcal{C}$ such that $\hat{F}$ is the identity map restricted to objects, automorphisms and unfactorizable morphisms in $\mathcal{C}$. This finite free EI category $\hat {\mathcal{C}}$ is unique up to isomorphism.
\item $\mathcal{C}$ is a finite free EI category if and only if so are all full subcategories.
\item The algebra $k\mathcal{C}$ is hereditary if and only if $\mathcal{C}$ is a finite free EI category and the automorphism group of each object has order invertible in $k$.
\end{enumerate}
\end{proposition}

\begin{proof}
See Propositions 2.9, 2.10 and Theorem 5.3 in \cite{Li1}. The category $\hat{\mathcal{C}}$ in (1) is called the \textit{free EI cover} of $\mathcal{C}$.
\end{proof}

From now on we let $\mathcal{C}$ be a finite free EI category with three objects. Then $\mathcal{C}$ has one of the following two descriptions.

\begin{equation*}
\xymatrix{ \mathcal{C}: & x \ar@(ld,lu)[]|{G} \ar @<1ex>[r] ^{H{\alpha}G } \ar@<-1ex>[r] ^{\ldots} & y \ar@(dr,dl)[]|{H} \ar @<1ex>[r] ^{L{\beta}H} \ar@<-1ex>[r] ^{\ldots} & z \ar@(rd,ru)[]|{L} & &
x \ar@(ld,lu)[]|{G} \ar @<1ex>[r] ^{H{\alpha}G } \ar@<-1ex>[r] ^{\ldots} & y \ar@(dr,dl)[]|{H} & z \ar@(rd,ru)[]|{L} \ar @<-1ex>[l] _{H{\beta}L} \ar@<1ex>[l] _{\ldots}}.
\end{equation*}
For the first category, $\mathcal{C} (x, z) = L \beta H \times_H H\alpha G$, where $\times_H$ is the biset product defined in \cite{Webb1}. Let $H_1 = \Stab_H (\alpha G)$, $H_2 = \Stab_H (L\beta)$ (or $H_2 = \Stab_H (\beta L)$ for the second one).

We need the following technical lemma to prove the main result of this subsection.

\begin{lemma}
Let $G$ be a finite group with cyclic Sylow $p$-subgroups. Let $G_1$ and $G_2$ be two proper subgroups. Suppose that $k \uparrow _{G_1}^G$ and $k \uparrow _{G_2}^G$ has no common composition factors except $k$. If $|G_1 \backslash G / G_2 | \geqslant 2$, then $P_k$ has only composition factors isomorphic to $k$, and the Scott modules of both permutation modules are quotient modules of $P_k$ and are not isomorphic to $k$.
\end{lemma}

\begin{proof}
Observe that
\begin{align*}
\dim_k \Hom_{kG} (k \uparrow _{G_1}^G, k \uparrow _{G_2}^G) & = \dim_k \Hom_{kG_1} (k, k \uparrow _{G_2}^G \downarrow _{G_1}^G)\\
& = \dim_k \Hom_{kG_1} (k, \bigoplus _{s \in G_1 \backslash G / G_2} (s \otimes k) \uparrow ^{G_1} _{G_1 \cap  sG_2 s^{-1}})
\end{align*}
which is exactly the number of double cosets in $G_1 \backslash G /G_2$.

Decompose $k \uparrow _{G_1}^G = M_1 \oplus \ldots \oplus M_m$ and $k \uparrow _{G_2}^G = N_1 \oplus \ldots \oplus N_m$ into indecomposable summands, and suppose that $M_1$ and $N_1$ are the Scott modules. By the given condition, $P_k$ is a uniserial module. Suppose that $P_k$ has a composition factor $S \ncong k$. We want to get a contradiction. Note that by this assumption, $\Top (P_k/k) \ncong k$.

Since $k \uparrow _{G_1}^G$ and $k \uparrow _{G_2}^G$ have no other common composition factors except $k$, for $2 \leqslant i \leqslant m$ and $1 \leqslant j \leqslant n$, $\Hom_{kH} (M_i, N_j) = 0$ because the simple summands of $\Top (M_i)$ are not isomorphic to $k$, so are not composition factors (up to isomorphism) of $N_j$. Therefore,
\begin{equation*}
\Hom_{kG} (k \uparrow _{G_1}^G, k \uparrow _{G_2}^G) = \bigoplus _{i=1}^m \bigoplus _{j=1}^n \Hom_{kG} (M_i, N_j) \cong \bigoplus _{j=1}^n \Hom_{kG} (M_1, N_j).
\end{equation*}
Note that $M_1$ is a string module, and $\Top (P_k /k) \ncong k$. Therefore, for any $0 \neq \varphi \in \Hom_{kG} (M_1, N_j)$, if the image of $\varphi$ is not contained in $\Soc(N_j)$, we can get a common composition factor $S \ncong k$ of $M_1$ and $N_j$. This is not allowed by our assumption. Moreover, the image of $\varphi$ must be isomorphic to $k$ by the same reason. Consequently, we get $\Hom_{kG} (M_1, N_j) = 0$ for $2 \leqslant j \leqslant n$, and $\Hom_{kG} (M_1, N_1)$ is spanned by the morphism sending the composition factor $k$ in $\Top (M_1)$ onto the composition factor $k$ in $\Soc(N_1)$. So $\Hom_{kG} (k \uparrow _{G_1}^G, k \uparrow _{G_2}^G)$ is one-dimensional, contradicting the given condition. This contradiction tells us that $P_k$ only has composition factors isomorphic to $k$, and the principal block $B_0 (kG)$ has only one simple module $k$ lying in it. Thus both Scott modules are quotient modules of $P_k$. Moreover, since
\begin{equation*}
\dim_k \Hom_{kG} (M_1, N_1) = \dim_k \Hom_{kG} (k \uparrow _{G_1}^G , k \uparrow _{G_2}^G) = |G_1 \backslash G / G_2| \geqslant 2,
\end{equation*}
we conclude that neither of them is isomorphic to $k$.
\end{proof}

Now we can prove:

\begin{proposition}
Let $\mathcal{C}$ be one of the above two categories, and suppose that its category algebra is of finite representation type. If $\Char(k) = p \neq 2, 3$, then $H$ acts transitively on either $\mathcal{C} (x, y)$ or $\mathcal{D} (y, z)$ (or $\mathcal{C} (z, y)$ for the second one). Moreover, $| H_1 \backslash H /H_2| = 1$.
\end{proposition}

\begin{proof}
We only prove the first statement for the first category since the proof also works for the second one with a very small modification. Let $G_1 = \Stab_G (H\alpha)$ and $L_1 = \Stab_L (\beta H)$. Suppose that $H$ does not act transitively on either $\mathcal{C} (x, y)$ or $\mathcal{C} (y, z)$. By considering the full subcategory $\mathcal{E}$ with objects $x$ and $y$ we conclude that $G$ acts transitively on $\mathcal{C} (x, y)$. Similarly, $L$ acts transitively on $\mathcal{C} (y, z)$. Therefore, $G_1$ and $L_1$ are proper subgroups of $G$ and $L$ respectively. We claim that there is an indecomposable projective $kG$-module $P_S$ such that $S \ncong k_G$ and $\Top (P_S \downarrow _{G_1}^G)$ has a simple summand isomorphic to $k$. Dually, there is an indecomposable projective $kL$-module $P_T$ such that $T \ncong k_L$ and $\Soc (P_T \downarrow _{L_1}^L)$ has a simple summand isomorphic to $k$.

We prove the first claim since the second one is a dual statement. By Lemma 3.14, it is equivalent to saying that $k \uparrow _{G_1}^G$ has a composition factor $S \ncong k_G$, which is obviously true for $p = 0$. Thus we let $p \geqslant 5$. If $k \uparrow _{G_1}^G$ has only composition factors isomorphic to $k$, then it is indecomposable. By Proposition 3.12, either $|G_1|$ is invertible or $G_1$ contains a Sylow $p$-subgroup of $G$. The second case cannot happen. Otherwise $k \uparrow _{G_1}^G \cong k$, which is absurd since $G_1$ is a proper subgroup of $G$. Therefore, $|G_1|$ must be invertible in $k$ and $k \uparrow _{G_1}^G \cong P_k$. Moreover, $G$ has a nontrivial Sylow $p$-subgroup since $\dim_k P_k = |G: G_1| > 1$. Consequently, $\dim_k \End _{kG} (P_k) = \dim_k P_k \geqslant p \geqslant 5$, so $k\mathcal{E}$ is of infinite representation type by Proposition 3.17. This is not allowed, and the claim is proved.

For an indecomposable representation $R$ of the following quiver with dimension vector $(d_1, d_2, d_3, d_4, d_5)$,
\begin{equation*}
\xymatrix{ & 1 \ar[d] & & & k^{d_1} \ar[d] ^{\varphi_1} & & \\
2 \ar[r] & 5 \ar[r] \ar[d] & 4 \quad \ar@{=>}[r]^R & \quad k^{d_2} \ar[r] ^{\varphi_2} & k^{d_5} \ar[r] ^{\varphi_4} \ar[d] ^{\varphi_3} & k^{d_4} & \\
 & 3 & & & k^{d_3} & }
\end{equation*}
we can construct an indecomposable corresponding representation $\tilde{R}$ of $\mathcal{C}$ as follows: $\tilde{R} (x) \cong k^{d_1} \oplus P_S^{d_2}$, $\tilde{R} _(y) \cong k^{d_5}$, $\tilde{R} (z) \cong k^{d_3} \oplus P_T^{d_4}$. Linear maps $\tilde{R} (\alpha)$ and $\tilde{R} (\beta)$ can be defined in a way as we did in the proofs of Propositions 3.9 and 3.12. Consequently, $k\mathcal{C}$ is of infinite representation type. This contradiction shows that $H$ must act transitively on at least one biset, proving the first statement.\\

Now we turn to the proof of the second statement. There are three cases:\\

\textbf{Case I:} $\mathcal{C}$ is the first category. We consider the full subcategory $\mathcal{E}$ with objects $x$ and $z$. Note that $\mathcal{E} (x, z) = L \beta H \alpha G$. Since $k\mathcal{E}$ is of finite representation type as well, either $L$ or $G$ should act transitively on $\mathcal{E} (x, z)$. Without loss of generality suppose that $L$ acts transitively, i.e., $L \beta H \alpha G = L \beta \alpha$.

Take an arbitrary $h \in H$ and consider $\beta h \alpha$. Since $\beta h \alpha \in \mathcal{E} (x, y) = L \beta \alpha$, we can find some $l \in L$ such that $\beta h \alpha = l \beta \alpha$. Consequently, this non-isomorphism in $\mathcal{C} (x, z)$ has two decompositions into unfactorizable morphisms. By the unique factorization property, there should exist $h_1 \in H$ such that the following diagram commutes:
\begin{equation*}
\xymatrix {x \ar[r] ^{\alpha} \ar[d]^1 & y \ar[r]^{l\beta} \ar[d]^{h_1} & z \ar[d]^1\\
x \ar[r]^{h\alpha} & y \ar[r]^{\beta} & z}.
\end{equation*}

Consequently, $l \beta = \beta h_1$ and $h_1 \alpha = h \alpha$. So $h_1 \in H_2$ and $h^{-1} h_1 \in H_1$. Therefore, $h \in H_2 H_1$, and $|H_1 \backslash H /H_2| = 1$.\\

In the next two cases we consider the second category. Suppose that $|H_1 \backslash H /H_2| \geqslant 2$ (so $H_1 \neq H \neq H_2$), we want to get a contradiction.\\

\textbf{Case II:} $k \uparrow _{H_1}^H$ and $k \uparrow _{H_2}^H$ has a common composition factor $S \ncong k$. Then the projective $kH$-module $P_S \ncong P_k$ satisfies that $\Top (P_S \downarrow _{H_1}^H)$ and $\Soc (P_S \downarrow _{H_2}^H)$ have a summand isomorphic to $k_{H_1}$ and $k_{H_2}$ respectively. For an indecomposable representation $R$ of the following quiver with dimension vector $(d_1, d_2, d_3, d_4)$,
\begin{equation*}
\xymatrix{ & 2 & & & k^{d_2} & & \\
1 \ar[ur] \ar[dr] & & 4  \ar[ul] \ar[dl] \quad \ar@{=>}[r]^R & \quad k^{d_1} \ar[ur] ^{\varphi_1} \ar[dr] ^{\varphi_3} & & k^{d_4} \ar[ul] ^{\varphi_4}  \ar[dl] ^{\varphi_3} & \\
 & 3 & & & k^{d_3} & }
\end{equation*}
we can construct a corresponding representation $\tilde{R}$ of $\mathcal{C}$ as follows: $\tilde{R} (x) \cong k^{d_1}$, $\tilde{R} _(y) \cong k^{d_2} \oplus P_S^{d_3}$, $\tilde{R} (z) \cong k^{d_4}$. Linear maps $\tilde{R} (\alpha)$ and $\tilde{R} (\beta)$ can be defined in a way as we did before. We deduce that $k\mathcal{C}$ is of infinite representation type, contradicting the given condition.\\

\textbf{Case III:} $k \uparrow _{H_1}^H$ and $k \uparrow _{H_2}^H$ has no common composition factors except $k$. By the previous lemma, the projective $kH$-module $P_k$ and the Scott modules of both $k \uparrow _{H_1}^H$ and $k \uparrow _{H_2}^H$ only have composition factors isomorphic to $k$. Moreover, their dimensions are all bigger than 1.

Consider the following category $\mathcal{E}$:
\begin{equation*}
\xymatrix{ \mathcal{E}: & x \ar@(ld,lu)[]|{1} \ar @<1ex>[r] ^{H \bar{\alpha}} \ar@<-1ex>[r] ^{\ldots} & y \ar@(dr,dl)[]|{H} & z \ar@(rd,ru)[]|{1} \ar @<-1ex>[l] _{H\bar{\beta}} \ar@<1ex>[l] _{\ldots}},
\end{equation*}
where $\Stab_{H} (\bar{\alpha}) = H_1$ and $\Stab_{H} (\bar{\beta}) = H_2$. The reader can check that $\mathcal{E}$ is a quotient category of $\mathcal{C}$. So it suffices to prove the infinite representation type of $k\mathcal{E}$.

Let $e$ be a primitive idempotent in $kH$ such that $kHe \cong P_k$. Consider the algebra
\begin{equation*}
\Lambda = \End _{k\mathcal{E}} (k\mathcal{E} 1_x \oplus k\mathcal{E} e \oplus k\mathcal{E} 1_z) ^{\textnormal{op}} \cong (1_x + e + 1_z) k\mathcal{E} (1_x + e + 1_z).
\end{equation*}
It is easy to check that $1_x k\mathcal{E} 1_x \cong k \cong 1_z k\mathcal{E} 1_z$, $e k\mathcal{E} e \cong k[t] / (t^d)$ where $d = \dim_k P_k \geqslant 2$, and
\begin{equation*}
\dim_k e k\mathcal{E} 1_x = \dim_k e kH \alpha = \dim_k \Hom_{kH} (P_k, k \uparrow _{H_1}^H) = s \geqslant 2,
\end{equation*}
where $s$ is the dimension of the Scott module of $k\uparrow _{H_1}^H$. Similarly, $\dim_k e k\mathcal{E} 1_z = t \geqslant 2$, where $t$ is the dimension of the Scott module of $k\uparrow _{H_2}^H$. Therefore, $\Lambda$ is isomorphic to the path algebra of the following bounded quiver
\begin{equation*}
\xymatrix{ \bullet \ar[r]^{\gamma} & \bullet \ar@(dr,dl)[]|{\delta} & \bullet \ar[l]_{\mu}}
\end{equation*}
with relations $\delta^d = \delta^s \gamma = \delta^t \mu = 0$. Its covering is not locally representation-finite since it contains the following subquiver:
\begin{equation*}
\xymatrix{ \bullet \ar[r] & \bullet \ar[d] & \bullet \ar[l] \\ \bullet \ar[r] & \bullet & \bullet \ar[l]}
\end{equation*}
Therefore, $\Lambda$ and hence $k\mathcal{E}$ are of infinite representation type, contradicting the given condition.

\end{proof}

\end{document}